%% file: main.tex
\documentclass[submission,copyright,creativecommons]{eptcs}
\pdfoutput=1
\providecommand{\event}{ACT 2023}

\input{preamble}

\usepackage{string_diagrams}
\usepackage{xsavebox}
\usepackage{wrapfig}

\usetikzlibrary{backgrounds}
\pgfdeclarelayer{edgelayer}
\pgfdeclarelayer{nodelayer}
\pgfsetlayers{background,edgelayer,nodelayer,main}

\title{Optics for Premonoidal Categories}
\author{
James Hefford
\institute{University of Oxford, UK}
\email{james.hefford@cs.ox.ac.uk}
\and
Mario Román
\institute{Tallinn University of Technology, Estonia}
\email{mroman@ttu.ee}
}

\def\titlerunning{Optics for Premonoidal Categories}
\def\authorrunning{J. Hefford and M. Román}

\begin{document}

\maketitle

\begin{abstract}
  We further the theory of optics or ``circuits-with-holes'' to encompass premonoidal categories: monoidal categories without the interchange law.
  Every premonoidal category gives rise to an effectful category (i.e.\ a generalised Freyd-category) given by the embedding of the monoidal subcategory of central morphisms.
  We introduce ``pro-effectful'' categories and show that optics for premonoidal categories exhibit this structure.

  Pro-effectful categories are the non-representable versions of effectful categories, akin to the generalisation of monoidal to promonoidal categories.
  We extend a classical result of Day to this setting, showing an equivalence between pro-effectful structures on a category and effectful structures on its free tight cocompletion.
  We also demonstrate that pro-effectful categories are equivalent to prostrong promonads.
\end{abstract}

\section{Introduction}
Monoidal categories play a central role in many categorical models, from programming semantics \cite{asperti}, to quantum theory \cite{abramsky_coecke,coecke_kissinger_2017}, to electrical circuits \cite{bonchi_affine}, for they provide the necessary mathematical structure to describe the interaction of systems over time (by composition) and space (by tensor product).
Monoidal categories have a graphical calculus known as string diagrams \cite{joyal_street1} which provides a formalisation of circuit diagrams.
There has been much interest in studying the categorical structure of \textit{circuits-with-holes} \cite{chiribella_circuits,kissinger_caus,wilson_locality} over a given monoidal category $\cat{C}$; that is, incomplete diagrams in $\cat{C}$.

\begin{xlrbox}{optics_equiv}
  \begin{tikzpicture}[baseline={([yshift=-.8ex]current bounding box.center)}]
    \node[3leggedpants,wide] (pants) {};
    \node[tube,long,anchor=top] (tube) at (pants.leftleg) {};
    \node[tube,long,anchor=top] (tube2) at (pants.rightleg) {};
    \node[3leggedcopants,wide,anchor=leftleg] (copants) at (tube.bot) {};
    \node[label] (g) at ([yshift=0.2cm]pants.center) {$g$};
    \node[label] (f) at ([yshift=-0.2cm]copants.center) {$f$};
    \node[label] (u) at (copants.leftleg) {$u$};
    \node[label] (v) at (copants.rightleg) {$v$};
    \begin{pgfonlayer}{edgelayer}
      \draw ([xshift=-0.075cm,yshift=0.15cm]f.center) to [out=90,in=-90] (u) to (pants.leftleg) to [out=90,in=-90] ([xshift=-0.075cm,yshift=-0.15cm]g.center);
      \draw ([xshift=0.075cm,yshift=0.15cm]f.center) to [out=90,in=-90] (v) to (pants.rightleg) to [out=90,in=-90] ([xshift=0.075cm,yshift=-0.15cm]g.center);
      \draw (copants.midleg) to (f) to (copants.belt);
      \draw (pants.midleg) to (g) to (pants.belt);
    \end{pgfonlayer}
    \node[seam] at (tube.center) {};
    \node[seam] at (tube2.center) {};
    \node[system] at ([yshift=-0.15cm]copants.belt) {$a$};
    \node[system] at ([yshift=0.15cm]copants.midleg) {$b$};
    \node[system] at ([yshift=-0.15cm]pants.midleg) {$b'$};
    \node[system] at ([yshift=0.2cm]pants.belt) {$a'$};
  \end{tikzpicture}
\end{xlrbox}

\begin{xlrbox}{optics_equiv2}
  \begin{tikzpicture}[baseline={([yshift=-.8ex]current bounding box.center)}]
    \node[3leggedpants,wide] (pants) {};
    \node[tube,long,anchor=top] (tube) at (pants.leftleg) {};
    \node[tube,long,anchor=top] (tube2) at (pants.rightleg) {};
    \node[3leggedcopants,wide,anchor=leftleg] (copants) at (tube.bot) {};
    \node[label] (g) at ([yshift=0.2cm]pants.center) {$g$};
    \node[label] (f) at ([yshift=-0.2cm]copants.center) {$f$};
    \node[label] (u) at (pants.leftleg) {$u$};
    \node[label] (v) at (pants.rightleg) {$v$};
    \begin{pgfonlayer}{edgelayer}
      \draw ([xshift=-0.075cm,yshift=0.15cm]f.center) to [out=90,in=-90] (copants.leftleg) to (u) to [out=90,in=-90] ([xshift=-0.075cm,yshift=-0.15cm]g.center);
      \draw ([xshift=0.075cm,yshift=0.15cm]f.center) to [out=90,in=-90] (copants.rightleg) to (v) to [out=90,in=-90] ([xshift=0.075cm,yshift=-0.15cm]g.center);
      \draw (copants.midleg) to (f) to (copants.belt);
      \draw (pants.midleg) to (g) to (pants.belt);
    \end{pgfonlayer}
    \node[seam] at (tube.center) {};
    \node[seam] at (tube2.center) {};
    \node[system] at ([yshift=-0.15cm]copants.belt) {$a$};
    \node[system] at ([yshift=0.15cm]copants.midleg) {$b$};
    \node[system] at ([yshift=-0.15cm]pants.midleg) {$b'$};
    \node[system] at ([yshift=0.2cm]pants.belt) {$a'$};
  \end{tikzpicture}
\end{xlrbox}

\begin{wrapfigure}[11]{r}{0.35\textwidth}
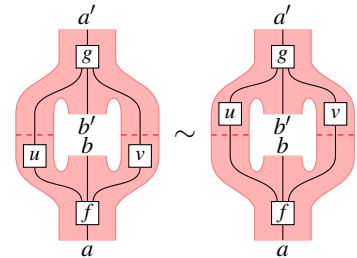

  \centering
  \xusebox{optics_equiv} $\ \sim \ $ \xusebox{optics_equiv2}
  \caption{Optics equivalence relation.}
  \label{fig:optics}
\end{wrapfigure}

The categorical methods required to describe these circuits-with-holes have their roots in the study of strong profunctors, or \textit{Tambara modules} \cite{tambara,boisseau_optics}.
These modules are the algebras for a certain promonad \cite{pastro_street} with the resulting category of free algebras known as the category of \textit{optics} by the functional programming community, where it is used to model various bidirectional data accessors \cite{clarke_profunctor,riley_optics,roman_optics}.
The category $\opt(\cat{C})$ has objects given by pairs $(a,a')$ of objects of $\cat{C}$ and homs $\opt(\cat{C})((a,a'),(b,b'))$ given by the quotiented sets of the form in Figure \ref{fig:optics} \cite{pastro_street,riley_optics,clarke_profunctor}.
The idea is to produce a category of holes in circuits from $\cat{C}$, where two circuits are equivalent if they can be rewritten into each other by sliding boxes.
This equivalence relation is handled by the coend $\int^{xy} \cat{C}(a,x\otimes b\otimes y) \times \cat{C}(x\otimes b'\otimes y,a')$.
Outside of functional programming, this category has been suggested as a way to model holes in general monoidal categories \cite{roman_coend}.
In particular, incomplete diagrams have applications in quantum theory where they are known as \textit{combs} \cite{chiribella_circuits} and capture a certain subset of the more general quantum supermaps \cite{chiribella_supermaps}: optics have been suggested to formalise these structures \cite{hefford_coend}.

\begin{xlrbox}{vertical_tensor}
  \begin{tikzpicture}[baseline={([yshift=-.8ex]current bounding box.center)}]
    \node[3leggedpants,wide] (pants) {};
    \node[tube,long,anchor=top] (tube) at (pants.leftleg) {};
    \node[tube,long,anchor=top] (tube2) at (pants.rightleg) {};
    \node[3leggedcopants,shortcrotch,wide,anchor=leftleg] (copants) at (tube.bot) {};
    \node[3leggedpants,shortcrotch,wide,anchor=belt] (pants2) at (copants.belt){};
    \node[tube,long,anchor=top] (tube3) at (pants2.leftleg) {};
    \node[tube,long,anchor=top] (tube4) at (pants2.rightleg) {};
    \node[3leggedcopants,wide,anchor=leftleg] (copants2) at (tube3.bot) {};
    \node[label] (h) at ([yshift=0.2cm]pants.center) {$h$};
    \node[label] (g) at (copants.belt) {$g$};
    \node[label] (f) at ([yshift=-0.2cm]copants2.center) {$f$};
    \begin{pgfonlayer}{edgelayer}
      \draw ([xshift=-0.075cm,yshift=0.15cm]g.center) to [out=90,in=-90] (copants.leftleg) to (pants.leftleg) to [out=90,in=-90] ([xshift=-0.075cm,yshift=-0.15cm]h.center);
      \draw ([xshift=0.075cm,yshift=0.15cm]g.center) to [out=90,in=-90] (copants.rightleg) to (pants.rightleg) to [out=90,in=-90] ([xshift=0.075cm,yshift=-0.15cm]h.center);
      \draw (copants.midleg) to (g) to (pants2.midleg);
      \draw (pants.midleg) to (h) to (pants.belt);
      \draw ([xshift=-0.075cm,yshift=0.15cm]f.center) to [out=90,in=-90] (copants2.leftleg) to (pants2.leftleg) to [out=90,in=-90] ([xshift=-0.075cm,yshift=-0.15cm]g.center);
      \draw ([xshift=0.075cm,yshift=0.15cm]f.center) to [out=90,in=-90] (copants2.rightleg) to (pants2.rightleg) to [out=90,in=-90] ([xshift=0.075cm,yshift=-0.15cm]g.center);
      \draw (copants2.midleg) to (f) to (copants2.belt);
    \end{pgfonlayer}
    % \node[seam] at (tube.center) {};
    % \node[seam] at (tube2.center) {};
    \node[system] at ([yshift=-0.15cm]pants2.midleg) {$a'$};
    \node[system] at ([yshift=0.15cm]copants2.midleg) {$a$};
    \node[system] at ([yshift=0.15cm]copants.midleg) {$b$};
    \node[system] at ([yshift=-0.15cm]pants.midleg) {$b'$};
    \node[system] at ([yshift=0.2cm]pants.belt) {$c'$};
    \node[system] at ([yshift=-0.2cm]copants2.belt) {$c$};
  \end{tikzpicture}
\end{xlrbox}

\begin{xlrbox}{horizontal_tensor}
  \begin{tikzpicture}[baseline={([yshift=-.8ex]current bounding box.center)}]
    \node[5leggedpants,wider] (pants) {};
    \node[tube,long,anchor=top] (tube) at (pants.leftleftleg) {};
    \node[tube,long,anchor=top] at (pants.midleg) {};
    \node[tube,long,anchor=top] at (pants.rightrightleg) {};
    \node[5leggedcopants,wider,anchor=leftleftleg] (copants) at (tube.bot) {};
    \node[label] (f) at ([yshift=-0.3cm]copants.center) {$f$};
    \node[label] (g) at ([yshift=0.3cm]pants.center) {$g$};
    \begin{pgfonlayer}{edgelayer}
      \draw (pants.belt) to (copants.belt);
      \draw ([xshift=-0.09cm]f.center) to [out=90,in=-90] (copants.leftleftleg) to (pants.leftleftleg) to [out=90,in=-90] ([xshift=-0.09cm]g.center);
      \draw ([xshift=0.09cm]f.center) to [out=90,in=-90] (copants.rightrightleg) to (pants.rightrightleg) to [out=90,in=-90] ([xshift=0.09cm]g.center);\
      \draw ([xshift=-0.05cm,yshift=0.1cm]f.center) to [out=90,in=-90] (copants.leftleg);
      \draw ([xshift=0.05cm,yshift=0.1cm]f.center) to [out=90,in=-90] (copants.rightleg);
      \draw ([xshift=-0.05cm,yshift=-0.1cm]g.center) to [out=-90,in=90] (pants.leftleg);
      \draw ([xshift=0.05cm,yshift=-0.1cm]g.center) to [out=-90,in=90] (pants.rightleg);
    \end{pgfonlayer}
    \node[system] at ([yshift=0.15cm]copants.leftleg) {$a$};
    \node[system] at ([yshift=-0.15cm]pants.leftleg) {$a'$};
    \node[system] at ([yshift=0.15cm]copants.rightleg) {$b$};
    \node[system] at ([yshift=-0.15cm]pants.rightleg) {$b'$};
    \node[system] at ([yshift=-0.15cm]copants.belt) {$c$};
    \node[system] at ([yshift=0.15cm]pants.belt) {$c'$};
  \end{tikzpicture}
\end{xlrbox}

What is still missing is a full description of optics for \textit{pre}monoidal categories.
Informally, premonoidal categories are like monoidal categories but dropping the interchange law so that in general $(1\otimes f)(g\otimes 1)\neq (g\otimes 1)(1\otimes f)$.
Such categories are useful in the modelling of computational side-effects and it was with precisely this motivation that Power and Robinson introduced them \cite{power_premonoidal}. This manuscript is indebted to the research into premonoidal categories \cite{jeffrey97,levy_push,power_structure,power02closed,staton_premulticategories}.

Since not all morphisms in a premonoidal category interchange, there is now an additional subtlety to formalising optics.
The coends that are usually used to quotient to allow for the sliding of morphisms can now only be taken over the \textit{centre} $Z\cat{C}$ of the premonoidal category $\cat{C}$ - the wide monoidal subcategory comprised of the morphisms that interchange with all others.

\begin{xlrbox}{premon_optics_equiv}
  \begin{tikzpicture}[baseline={([yshift=-.8ex]current bounding box.center)}]
    \node[3leggedpants,wide] (pants) {};
    \node[tube,long,bluetube,anchor=top] (tube) at (pants.leftleg) {};
    \node[tube,long,bluetube,anchor=top] (tube2) at (pants.rightleg) {};
    \node[3leggedcopants,wide,anchor=leftleg] (copants) at (tube.bot) {};
    \node[label] (g) at ([yshift=0.2cm]pants.center) {$g$};
    \node[label] (f) at ([yshift=-0.2cm]copants.center) {$f$};
    \node[label] (u) at (copants.leftleg) {$u$};
    \node[label] (v) at (copants.rightleg) {$v$};
    \begin{pgfonlayer}{edgelayer}
      \draw ([xshift=-0.075cm,yshift=0.15cm]f.center) to [out=90,in=-90] (u) to (pants.leftleg) to [out=90,in=-90] ([xshift=-0.075cm,yshift=-0.15cm]g.center);
      \draw ([xshift=0.075cm,yshift=0.15cm]f.center) to [out=90,in=-90] (v) to (pants.rightleg) to [out=90,in=-90] ([xshift=0.075cm,yshift=-0.15cm]g.center);
      \draw (copants.midleg) to (f) to (copants.belt);
      \draw (pants.midleg) to (g) to (pants.belt);
    \end{pgfonlayer}
    \node[seam] at (tube.center) {};
    \node[seam] at (tube2.center) {};
    \node[system] at ([yshift=-0.15cm]copants.belt) {$a$};
    \node[system] at ([yshift=0.15cm]copants.midleg) {$b$};
    \node[system] at ([yshift=-0.15cm]pants.midleg) {$b'$};
    \node[system] at ([yshift=0.2cm]pants.belt) {$a'$};
  \end{tikzpicture}
\end{xlrbox}

\begin{xlrbox}{premon_optics_equiv2}
  \begin{tikzpicture}[baseline={([yshift=-.8ex]current bounding box.center)}]
    \node[3leggedpants,wide] (pants) {};
    \node[tube,long,bluetube,anchor=top] (tube) at (pants.leftleg) {};
    \node[tube,long,bluetube,anchor=top] (tube2) at (pants.rightleg) {};
    \node[3leggedcopants,wide,anchor=leftleg] (copants) at (tube.bot) {};
    \node[label] (g) at ([yshift=0.2cm]pants.center) {$g$};
    \node[label] (f) at ([yshift=-0.2cm]copants.center) {$f$};
    \node[label] (u) at (pants.leftleg) {$u$};
    \node[label] (v) at (pants.rightleg) {$v$};
    \begin{pgfonlayer}{edgelayer}
      \draw ([xshift=-0.075cm,yshift=0.15cm]f.center) to [out=90,in=-90] (copants.leftleg) to (u) to [out=90,in=-90] ([xshift=-0.075cm,yshift=-0.15cm]g.center);
      \draw ([xshift=0.075cm,yshift=0.15cm]f.center) to [out=90,in=-90] (copants.rightleg) to (v) to [out=90,in=-90] ([xshift=0.075cm,yshift=-0.15cm]g.center);
      \draw (copants.midleg) to (f) to (copants.belt);
      \draw (pants.midleg) to (g) to (pants.belt);
    \end{pgfonlayer}
    \node[seam] at (tube.center) {};
    \node[seam] at (tube2.center) {};
    \node[system] at ([yshift=-0.15cm]copants.belt) {$a$};
    \node[system] at ([yshift=0.15cm]copants.midleg) {$b$};
    \node[system] at ([yshift=-0.15cm]pants.midleg) {$b'$};
    \node[system] at ([yshift=0.2cm]pants.belt) {$a'$};
  \end{tikzpicture}
\end{xlrbox}

\noindent
\begin{minipage}{0.5\textwidth}\centering
  \scalebox{0.9}{\xusebox{horizontal_tensor}} \ ,\hspace{0.5cm} \scalebox{0.9}{\xusebox{vertical_tensor}}
  \captionof{figure}{\textit{Left}: $\otimes_H$, \textit{Right}: $\otimes_V$.}
  \label{fig:optics_tensors}
\end{minipage}%
\begin{minipage}{0.5\textwidth}\centering
  \scalebox{0.9}{\xusebox{premon_optics_equiv}} $\ \sim \ $ \scalebox{0.9}{\xusebox{premon_optics_equiv2}}
  
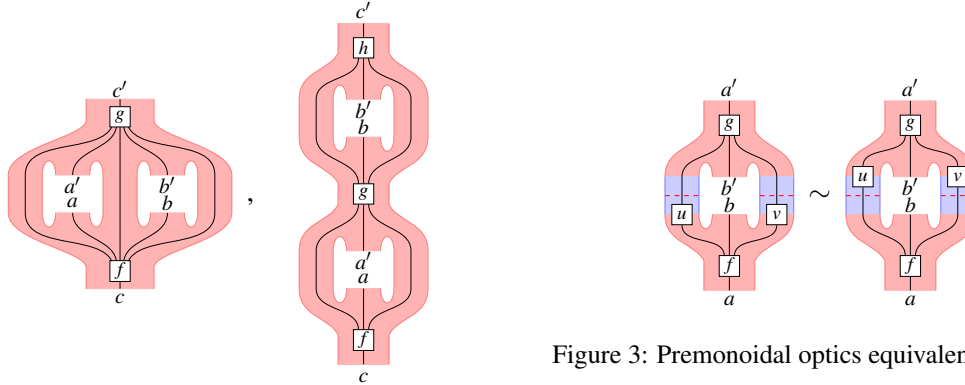
\captionof{figure}{Premonoidal optics equivalence relation.}
  \label{fig:premon_optics}
\end{minipage}

In this article, we will develop the machinery required to deal with optics for premonoidal categories.
There is a category $\opt_{Z\cat{C}}(\cat{C})$ with objects given by pairs $(a,a')$ of those of $\cat{C}$ and homs given by the sets of the form in Figure \ref{fig:premon_optics}, handled by the coend $\int^{xy\in Z\cat{C}} \cat{C}(a,x\otimes b\otimes y) \times \cat{C}(x\otimes b'\otimes y,a')$.
While premonoidal optics can be seen to be a special case of generalised optics for an actegory \cite{clarke_profunctor,riley_optics}, the full monoidal-like structure of the category has not been discussed before.
Optics over a monoidal category $\cat{C}$ are equipped with two promonoidal tensor products, $\otimes_H,\otimes_V:\opt(\cat{C})\times\opt(\cat{C})\pmorph\opt(\cat{C})$, which capture the horizontal and vertical composition of holes (see Figure \ref{fig:optics_tensors}).
These two tensors interact to make $\opt(\cat{C})$ into a produoidal category \cite{earnshaw}.

Over a premonoidal category $\cat{C}$ we might hope to equip $\opt_{Z\cat{C}}(\cat{C})$ with two tensors analogous to those in Figure \ref{fig:optics_tensors}.
While the vertical tensor $\otimes_V$ poses no immediate difficulties, the horizontal tensor $\otimes_H$ does: we cannot expect this to be promonoidal because $\cat{C}$ does not satisfy interchange.
This requires us to introduce the notion of a \textit{pro-effectful} category which combines the structures of premonoidal and promonoidal categories together -- this is our main technical contribution.

We prove that pro-effectful categories are equivalently: \emph{(i)} prostrong promonads; \emph{(ii)} biproactegories (two-sided actions in the category of profunctors) which suitably extend a canonical action on the centre of the category; and \emph{(iii)} pseudomonoids in the bicategory of tight $\vsq$-profunctors.

Each of these gives a different perspective on pro-effectful categories, connecting them, respectively, with monads; the action definition of Freyd-categories given by Levy \cite{levy_push}; the pseudomonoid definition of effectful categories given by Román \cite{roman_promonads} and the work on closed effectful categories due to Power \cite{power_structure,power_generic}.
In particular, this final perspective demonstrates that pro-effectful categories are equivalent to closed effectful categories on the free tight cocompletion, where the effectful structure is given by a version of Day convolution.

Finally, there is an additional challenge with premonoidal optics: there is the category $\opt(Z\cat{C})$ of optics over the monoidal centre and an embedding, $\opt(Z\cat{C})\morph{}\opt_{Z\cat{C}}(\cat{C})$, of these central optics into the optics over the entire premonoidal category.
$\opt(Z\cat{C})$ is equipped with the two promonoidal structures, $\otimes_H$ and $\otimes_V$, and we would like to understand how these behave in relation to any tensors we can define on $\opt_{Z\cat{C}}(\cat{C})$.
This requires us to keep track of the centre and understand fully how it behaves in relation to the rest of the premonoidal category.

\section{Premonoidal and Effectful Categories}\label{sec:premon_effectful}
% \textit{Note on string diagrams:} We make use of string diagrams for the monoidal bicategory $\VProf$ of categories, profunctors and natural transformations. In particular we use the ``wire diagrams'' of \cite{bartlett_wirediagrams} so that 1-cells are boxes (to be read bottom-to-top) and the monoidal product of 1-cells places boxes side-by-side. 2-cells are represented as arrows between diagrams.

Let us start by formally defining premonoidal categories enriched over a fixed cosmos $\mathcal{V}$, taken to be bicomplete and closed symmetric monoidal.
We take some space to spell this out as there are some technicalities involved which do not appear to have been explicitly discussed elsewhere.
Unless otherwise indicated ``category,'' ``functor,'' ``natural transformation'' etc.\ should be taken to mean $\cosmos$-category etc.
We write $\otimes$ for the tensor of $\cosmos$ and for the enriched tensor of categories \cite{kelly}.

\begin{defn}[Binoidal Category]\label{defn:binoidal}
  A category $\cat{C}$ is binoidal when, for each object $a$, it is equipped with a pair of functors $a\ltimes -:\cat{C}\morph{}\cat{C}$ and $-\rtimes a:\cat{C}\morph{}\cat{C}$ such that for all $a$ and $b$, $a\ltimes b = a\rtimes b$.
\end{defn}

\begin{remark}
  In the case of $\cosmos=\set$, the previous definition is equivalent to the one in terms of the funny tensor product \cite{foltz,weber}, though this must be avoided in the enriched case because it relies on the discrete category $\cat{C}_0$ of objects of a category $\cat{C}$ which is ill-defined over arbitrary $\cosmos$.
\end{remark}
We now generalize the notion of \emph{central} natural transformation to the enriched case.

\begin{defn}[Central Natural Transformation]
  Let $\cat{C}$ be a binoidal category and $F,G:\cat{D}\morph{}\cat{C}$ be two functors.
  Let $\eta:F\morph{}G$ be a natural transformation, so that we have a family of morphisms $\eta_a:I_\cosmos\morph{}\cat{C}(Fa,Ga)$ of $\cosmos$ satisfying the usual naturality diagrams.
  The component $\eta_a$ is central when the following diagram commutes for all objects $c$ and $d$, and $\eta$ is central when all components are central.
  \begin{equation*}
    \begin{tikzcd}[column sep=2.8em,cramped]
      \cat{C}(c,d)\otimes I_\cosmos \ar[r,"1\otimes \eta_a"] & \cat{C}(c,d)\otimes\cat{C}(Fa,Ga) \ar[rr,"(Ga\ltimes-)\otimes(-\rtimes c)"] & & \cat{C}(Ga\ltimes c,Ga\ltimes d)\otimes\cat{C}(Fa\rtimes c,Ga\rtimes c) \ar[d,"\circ"] \\
      \cat{C}(c,d) \ar[u,"\rho"] \ar[d,"\lambda"'] & & & \cat{C}(Fa\rtimes c,Ga\rtimes d) \\
      I_\cosmos \otimes \cat{C}(c,d) \ar[r,"\eta_a\otimes 1"'] & \cat{C}(Fa,Ga)\otimes\cat{C}(c,d) \ar[rr,"(-\rtimes d)\otimes(Fa\ltimes-)"'] & & \cat{C}(Fa\rtimes d,Ga\rtimes d)\otimes\cat{C}(Fa\ltimes c,Fa\ltimes d) \ar[u,"\circ"' {yshift=-0.5mm}] 
    \end{tikzcd}
  \end{equation*}
\end{defn}

Binoidal categories give us the necessary machinery to define premonoidal categories.
\begin{defn}[Premonoidal Category]
  A premonoidal category $\cat{C}$ is a binoidal category endowed with an object $i$ and central natural isomorphisms,
  $(a\otimes b)\otimes c \cong a\otimes(b\otimes c)$ and $a\otimes i\cong a\cong i\otimes a$,
  such that the triangle and pentagon equations hold.
\end{defn}

As we have just seen, enriched premonoidal categories are require some care to define formally - since the coherence isomorphisms need to be central we are required to define binoidal categories first so that we can make sense of this centrality.
Even in the case $\cosmos=\set$, the 2-category $\Cat$ fails to be monoidal under the funny tensor product because the funny tensor of natural transformations is not well-defined unless the components are all central \cite{power_premonoidal}.
This prevents the swift and elegant definition ``\textit{a premonoidal category is a pseudomonoid in} $\Cat_{\funny}$.''

Power realised that premonoidal categories are more algebraically well-defined when one shifts to working with premonoidal categories with a specified subcategory of central morphisms \cite{power_structure}.
We call these \emph{effectful categories} \cite{roman_promonads}.
In order to define them we first need to give the definition of a functor between premonoidal categories.

\begin{defn}[Centre Piece]
  Let $\cat{C}$ be a binoidal category.
  A centre piece at objects $(a,b)$ in $\cat{C}$ is an object $U(a,b)$ in $\cosmos$, endowed with an arrow $\iota:U(a,b)\morph{}\cat{C}(a,b)$, such that for any objects $(c,d)$ the following diagrams commute.
  \begin{equation*}
    \begin{tikzcd}[cramped]
      U(a,b)\otimes \cat{C}(c,d) \ar[r,"\iota\otimes 1"] \ar[d,"\iota\otimes 1"'] & \cat{C}(a,b)\otimes \cat{C}(c,d) \ar[r,"(-\rtimes c)\otimes(b\ltimes -)"]& \cat{C}(a\rtimes c,b\rtimes c)\otimes \cat{C}(b\ltimes c,b\ltimes d) \ar[d,"\circ_\sigma"]\\
      \cat{C}(a,b)\otimes \cat{C}(c,d) \ar[r,"(-\rtimes d)\otimes(a\ltimes -)"'] & \cat{C}(a\rtimes d,b\rtimes d)\otimes \cat{C}(a\ltimes c,a\ltimes d) \ar[r,"\circ"'] & \cat{C}(a\rtimes c,b\rtimes d)
    \end{tikzcd}
  \end{equation*}
  \begin{equation*}
    \begin{tikzcd}[cramped]
      \cat{C}(c,d) \otimes U(a,b) \ar[r,"1\otimes\iota"] \ar[d,"1\otimes \iota"'] & \cat{C}(c,d)\otimes \cat{C}(a,b) \ar[r,"(-\rtimes a)\otimes(d\ltimes -)"]& \cat{C}(c\rtimes a,d\rtimes a)\otimes \cat{C}(d\ltimes a,d\ltimes b) \ar[d,"\circ_\sigma"]\\
      \cat{C}(c,d)\otimes \cat{C}(a,b) \ar[r,"(-\rtimes b)\otimes(c\ltimes -)"'] & \cat{C}(c\rtimes b,d\rtimes b)\otimes \cat{C}(c\ltimes a,c\ltimes b) \ar[r,"\circ"'] & \cat{C}(c\rtimes a,d\rtimes b)
    \end{tikzcd}
  \end{equation*}
  % A morphism of centre pieces at $(a,b)$ is an arrow $u:V(a,b)\morph{}U(a,b)$ in $\cosmos$ such that $\iota_U \circ u = \iota_V$.
  % Centre pieces and morphisms of centre pieces form a category $CP(a,b)$.
\end{defn}

% \begin{defn}[Centre]
%   Let $\cat{C}$ be binoidal.
%   The centre of $\cat{C}$ at objects $(a,b)$ is the universal centre piece, $\iota:Z\cat{C}(a,b)\morph{}\cat{C}(a,b)$, such that all other centre pieces factorise uniquely through it.
% \end{defn}
% Universal centre pieces assemble to give a category $Z\cat{C}$ with the same objects as $\cat{C}$ and with hom-objects given by the universal centre pieces.
% Composition and identities are inherited from $\cat{C}$ because \eqref{eq:inducedcentrepiece} and $j_a$ are centre pieces and thus factor via the centre.

% \begin{prop}\label{prop:centre_is_cat}
%   The arrow \eqref{eq:inducedcentrepiece}, representing composition inherited from $\cat{C}$, is a center piece at $(a,c)$.
%   \begin{equation}\label{eq:inducedcentrepiece}
%     Z\cat{C}(b,c)\otimes Z\cat{C}(a,b) \morph{\iota\otimes\iota} \cat{C}(b,c)\otimes\cat{C}(a,b) \morph{\circ} \cat{C}(a,c)
%   \end{equation}
%    Furthermore, for each object $a$, the arrow $j_a:I_\cosmos\morph{}\cat{C}(a,a)$, representing the identities, is a centre piece at $(a,a)$.
%    As a result, $Z\cat{C}$ is a category.
% \end{prop}
% \begin{proof}
%   In Appendix \ref{proof:centre_is_cat}.
% \end{proof}
% Finally we note that the $\iota$ assemble to give an identity on objects functor $Z\cat{C}\morph{}\cat{C}$.

\begin{defn}[Premonoidal Functor]
  A premonoidal functor $F:\cat{C}\morph{}\cat{D}$ between premonoidal categories is a functor which preserves the centre pieces of $\cat{C}$ such that for all centre pieces $\iota:U(a,b)\morph{}\cat{C}(a,b)$, composition with $F$ gives a centre piece $U(a,b) \morph{\iota} \cat{C}(a,b) \morph{F_{a,b}} \cat{D}(Fa,Fb)$ at $(Fa,Fb)$ in $\cat{D}$.
  Furthermore $F$ must preserve the premonoidal structure up to natural transformations $Fa\otimes Fb \morph{}F(a\otimes b)$ and $i\morph{}Fi$ subject to coherence conditions like those for a monoidal functor.
  A premonoidal functor is strict when these transformations are identities.
\end{defn}

\begin{defn}[Effectful Category]
  An effectful category consists of a monoidal category $\cat{C}_0$, a premonoidal category $\cat{C}_1$ with the same objects as $\cat{C}_0$ and a strict, identity on objects, premonoidal functor $J:\cat{C}_0\morph{}\cat{C}_1$.
\end{defn}
\begin{remark}
  If $\cat{C}_0$ is cartesian, then an effectful category is known as a \emph{Freyd category} \cite{staton_premulticategories,levy_push}.
  At least one good reason to relax from requiring cartesian structure to arbitrary monoidal structure concerns applications of effectful categories outside of computer science.
  For instance, there has been interest in effectful structure for models of spacetime and quantum theory where products are not the canonical way of taking joint systems \cite{comeau,hefford_spacetime}, and in the study of Petri nets \cite{baez_petri}.
\end{remark}

Effectful categories can be seen as particular instances of actegories \cite{levy_push} - that is, a category with an action by a monoidal category \cite{benabou_bicategories,capucci_actegories}.
An effectful category $J:\cat{C}_0\morph{}\cat{C}_1$ specifies a left and right $\cat{C}_0$-action on $\cat{C}_1$, making $\cat{C}_1$ into a $\cat{C}_0$-$\cat{C}_0$-biactegory.
These actions $\ltimes:\cat{C}_0\otimes\cat{C}_1\morph{}\cat{C}_1$ and $\rtimes:\cat{C}_1\otimes\cat{C}_0\morph{}\cat{C}_1$, are special because they preserve the canonical actions of $\cat{C}_0$ on itself, i.e.\ $J$ extends the action $J\boxtimes = \ltimes(1\otimes J)$ and $J\boxtimes = \rtimes(J\otimes 1)$, and preserves the coherence isomorphisms.

Effectful categories are also precisely the same thing as strong promonads \cite{jacobs_arrows,garner,roman_promonads}.
Given a strong promonad $T:\cat{C}\pmorph\cat{C}$ there is a canonical premonoidal structure on the Kleisli category $\kl{T}$ which on objects acts like the tensor of $\cat{C}$.
The free functor $F:\cat{C}\morph{}\kl{T}$ then constitutes an effectful category.
Conversely, given an effectful category $J:\cat{C}_0\morph{}\cat{C}_1$, there is a promonad $\cat{C}_1(J-,J-):\cat{C}_0\pmorph\cat{C}_0$ which can be shown to be strong with strengths induced by the action of $\cat{C}_0$ on $\cat{C}_1$.

\subsubsection*{Effectful Categories as Pseudomonoids}
In this section, we show that effectful categories are pseudomonoids in the category of $\vsq$-enriched categories equipped with a modified version of the funny tensor product, where $\vsq = [\morph{},\cosmos]$ is the category of arrows and commutative squares in $\cosmos$.
In doing so we place effectful categories on the same footing as monoidal and promonoidal categories, showing that they are representations of the same underlying algebraic data.
This builds upon the work of Power who first studied the algebraicity of effectful categories in $\vsqCat$ \cite{power_structure,power_generic}.

\begin{prop}\label{prop:v2cosmos}
  Let $\cosmos$ be a complete, cocomplete, closed symmetric monoidal category.
  Then $\vsq$ is also a complete, cocomplete, closed symmetric monoidal category and therefore constitutes a cosmos.
\end{prop}
% \begin{proof}
%   In Appendix \ref{proof:v2cosmos}.
% \end{proof}

Since $\vsq$ is a cosmos, we can consider categories enriched in $\vsq$ \cite{power_generic}.
A $\vsq$-category consists of a pair of categories $\cat{C}_0$ and $\cat{C}_1$ with the same objects, and an identity on objects functor $J:\cat{C}_0\morph{}\cat{C}_1$.
A $\vsq$-functor $F:J_\cat{C}\morph{}J_\cat{D}$ consists of a pair of functors $F_0:\cat{C}_0\morph{}\cat{D}_0$ and $F_1:\cat{C}_1\morph{}\cat{D}_1$ such that $F_1 J_{\cat{C}} = J_{\cat{D}} F_0$.
A $\vsq$-natural transformation $\eta:F\Rightarrow G$ between $\vsq$-functors $F,G:J_\cat{C}\morph{}J_\cat{D}$ consists of natural transformations $\eta^0:F_0\Rightarrow G_0$ and $\eta^1:F_1\Rightarrow G_1$ with components that satisfy $J_\cat{D}(\eta^0_c)=\eta^1_c$.
When $J_\cat{D}$ is an embedding, we can think of this transformation simply as having central components, in $\cat{D}_0$.

There is a bicategory $\vsqCat$ of $\vsq$-categories, $\vsq$-functors and $\vsq$-natural transformations.
This bicategory has an interesting tensor that arises as a slight modification of the funny tensor product.

\begin{defn}[Funny Tensor of $\vsq$-Categories]
  Given two $\vsq$-categories $J_\cat{C}$ and $J_\cat{D}$, their funny tensor $J_{\cat{C}\funny\cat{D}}:\cat{C}_0\otimes\cat{D}_0\morph{}\cat{C}_1\funny\cat{D}_1$ is the identity on objects functor given by the diagonal of the following pushout in $\VCat$.
  \begin{equation}\label{eq:funny}
    \begin{tikzcd}[cramped]
      \cat{C}_0\otimes\cat{D}_0 \ar[r,"J_\cat{C}\otimes 1"] \ar[d,"1\otimes J_\cat{D}"']& \cat{C}_1\otimes\cat{D}_0 \ar[d,"i_0"] \\
      \cat{C}_0\otimes\cat{D}_1 \ar[r,"i_1"'] & \cat{C}_1\funny\cat{D}_1 \arrow[ul, phantom, "\ulcorner", very near start]
    \end{tikzcd}
  \end{equation}
  The pushout exists because $\cosmos$ is cocomplete and thus $\VCat$ is also cocomplete \cite{wolff}.
  Given $\vsq$-functors $F:J_\cat{A}\morph{}J_\cat{B}$ and $G:J_\cat{C}\morph{}J_\cat{D}$ their funny tensor $F\funny G$ has components $(F\funny G)_0 = F_0\otimes G_0$ and $(F\funny G)_1=F_1\funny G_1$ given by the unique arrow induced by the pushout.
  The funny tensor is also well-behaved on $\vsq$-natural transformations because their components $J_\cat{D}(\eta^0_c)=\eta^1_c$ are central and thus interchange with all other morphisms in $\cat{C}\funny\cat{D}$.
\end{defn}

\begin{thm}\label{thm:v2cat}
  $\vsqCat$ is a monoidal 2-category under the funny tensor $\funny$.
\end{thm}
\begin{proof}[Proof sketch]
  In Appendix \ref{proof:v2cat}.
\end{proof}

% In fact, as a 1-category, $\vsqCat$ is closed monoidal.
% \begin{prop}
%   $\vsqCat$ is a closed monoidal category where the internal-hom is given by the inclusion $[\cat{C},\cat{D}]\morph{}[\cat{C},\cat{D}]_u$ of the category of $\vsq$-functors and $\vsq$-natural transformations into the category of $\vsq$-functors and $\vsq$-unnatural transformations.
% \end{prop}

This leads to the main theorem of this section.
Theorem \ref{thm:eff_pseudo} is equivalent to the result of Román \cite{roman_promonads}: effectful categories are pseudomonoids in the bicategory of promonads, promonad homomorphisms and promonad modifications.
\begin{thm}\label{thm:eff_pseudo}
  An effectful category is a pseudomonoid in $\vsqCat_{\funny}$.
\end{thm}
\begin{proof}[Proof sketch]
  A pseudomonoid in $\vsqCat_{\funny}$ consists of a $\vsq$-category $J:\cat{C}_0\morph{}\cat{C}_1$ equipped with $\vsq$-functors $\boxtimes:J\funny J\morph{}J$ and $I:1\morph{}J$, such that there are $\vsq$-natural isomorphisms $\boxtimes(\boxtimes\otimes 1)\overset{\alpha}{\cong}\boxtimes(1\otimes\boxtimes)\mbox{ and } \boxtimes(I\otimes 1)\overset{\lambda}{\cong} 1 \overset{\rho}{\cong} \boxtimes(1\otimes I)$.

  Note that $\boxtimes$ consists of two functors $\boxtimes_0:\cat{C}_0\otimes\cat{C}_0\morph{}\cat{C}_0$ and $\boxtimes_1:\cat{C}_1\funny\cat{C}_1\morph{}\cat{C}_1$ such that $J\boxtimes_0 = \otimes_1J_{\cat{C}\funny\cat{C}}$.
  $\boxtimes_0$ together with $I_0$ and the natural isomorphisms $\alpha_0,\rho_0$ and $\lambda_0$, give a monoidal structure on $\cat{C}_0$.

  The $\cat{C}_0$-biaction on $\cat{C}_1$ is given by the compositions $\ltimes := \boxtimes i_1$ and $\rtimes := \boxtimes i_0$.
  That $J$ preserves the canonical actions given by $\otimes_0$ on $\cat{C}_0$ follows by the diagram \eqref{eq:funny} and the equality $J\otimes_0 = \otimes_1J_{\cat{C}\funny\cat{C}}$, together with the fact that $\alpha_1,\rho_1$ and $\lambda_1$ have components in the image of $J$.
  The coherence equations of the biaction are a consequence of those of $\alpha_1,\rho_1$ and $\lambda_1$: for instance $\alpha_1$ is a natural isomorphism between functors with type $\cat{C}_1\funny\cat{C}_1\funny\cat{C}_1\morph{}\cat{C}_1$.\
  This amounts to ``separate'' naturality in each $\cat{C}_1$ of the domain which in turns induces the left, bimodule and right coherences for the biaction.
\end{proof}

\begin{thm}
 There is an equivalence of bicategories $\vsqCat_{\funny}\cong\cosmos\text{-}\mathsf{Promonad}$ between the bicategories of $\vsq$-categories under the funny tensor product and the bicategory of promonads.
\end{thm}
\begin{proof}[Proof sketch]
  The result follows upon unwinding the definitions \cite{roman_promonads} and comparing the rest of the definitions there with those of the present section. Promonads use the ``pure tensor'' in \cite{roman_promonads}.
\end{proof}

\section{Closed Effectful Categories}\label{sec:closed_premon}
Now that we have a thorough understanding of effectful categories, we can start to work towards their ``pro-'' analogue.
To start, recall that a promonoidal category is equivalently a \textit{closed} monoidal presheaf category.
This suggests we should turn our attention to the closure of effectful categories, which will be the focus of this section.

Power gave the following definition of closure for effectful categories, where there is still an adjunction between tensoring and the internal-hom, but only for the centre \cite{power_structure}.

\begin{defn}[Closed Effectful Category \cite{power_structure}]
  An effectful category $J:\cat{C}_0\morph{}\cat{C}_1$ is right-closed when for each object $X$, $J(-)\otimes X:\cat{C}_0\morph{}\cat{C}_1$ has a right adjoint $[X,-]:\cat{C}_1\morph{}\cat{C}_0$.
  An effectful category is left-closed when for each $X$, $X\otimes J(-):\cat{C}_0\morph{}\cat{C}_1$ has a right adjoint.
  We say an effectful category is closed if it is both left and right-closed.
\end{defn}

Power proved the following result which generalises Day's result that every monoidal category embeds into a closed monoidal category \cite{day}.

\begin{thm}[\cite{power_structure}]
  Every (small) effectful category embeds into a closed effectful category.
\end{thm}

We say that an effectful category $J:\cat{C}_0\morph{}\cat{C}_1$ is \textit{small} when both $\cat{C}_0$ and $\cat{C}_1$ are small.
Given small $J$, we can take the strong promonad $T(-,-):=\cat{C}_1(J-,J-):\cat{C}_0\pmorph\cat{C}_0$ and lift it to a strong cocontinuous monad on the presheaf category $\widehat{T}:\widehat{\cat{C}_0}\morph{}\widehat{\cat{C}_0}$.
The Kleisli category $\kl{\widehat{T}}$ has as objects presheaves $F:\opcat{\cat{C}_0}\morph{}\cosmos$ and homs $\kl{\widehat{T}}(F,G) = \widehat{\cat{C}_0}(F,\widehat{T}G)$.
Moreover, $\widehat{\cat{C}_0}$ is monoidal under Day convolution while $\kl{\widehat{T}}$ is premonoidal.
As a result there is an effectful category given by the identity on objects functor $\widehat{\cat{C}_0}\morph{}\kl{\widehat{T}}$.

Power gave another characterisation of the effectful category $\widehat{\cat{C}_0}\morph{}\kl{\widehat{T}}$ as the free \textit{tight} cocompletion of the $\vsq$-category $J:\cat{C}_0\morph{}\cat{C}_1$ - that is, the cocompletion in only $\cosmos$-colimits, not all $\vsq$-colimits.
In the case of $\cosmos=\set$  these are precisely the ``conical'' colimits.
The name ``tight'' was first suggested in \cite{lack_shulman} where the theory of categories enriched in $\Cat^2$ is studied in some detail.

\begin{thm}[\cite{power_generic}]
  The free tight cocompletion of a small $\vsq$-category $J:\cat{C}_0\morph{}\cat{C}_1$ is the bijective on objects functor $\lan_{\opcat{J}}^L:\widehat{\cat{C}_0}\morph{}\overline{\cat{C}_1}$ induced by the functor $\lan_{\opcat{J}}: \widehat{\cat{C}_0}\morph{}\widehat{\cat{C}_1}$, via its canonical factorisation into a bijective on objects functor followed by a fully faithful functor (its bo-ff factorisation).
\end{thm}

The category $\overline{\cat{C}_1}:=\Im(\lan_{\opcat{J}})$ given by the bo-ff factorisation of $\lan_{\opcat{J}}$ has as objects presheaves $F:\opcat{\cat{C}_0}\morph{}\cosmos$ and as homs $\overline{\cat{C}_1}(F,G) = \widehat{\cat{C}_1}(\lan_{\opcat{J}} F, \lan_{\opcat{J}} G)$.
By the adjunction between extension and restriction of presheaves along $J$ there is a natural isomorphism
\begin{equation*}
  \kl{\widehat{T}}(F,G) = \widehat{\cat{C}_0}(F,\widehat{T}G) \cong \widehat{\cat{C}_1}(\lan_{\opcat{J}} F, \lan_{\opcat{J}} G) = \overline{\cat{C}_1}(F,G)
\end{equation*}
Thus to give a natural transformation $F \Rightarrow \widehat{T}G$ is equivalent to giving one $\lan_{\opcat{J}}F \Rightarrow \lan_{\opcat{J}}G$.
This demonstrates an isomorphism $\overline{\cat{C}_1} \cong \kl{\widehat{T}}$

As a consequence, the following diagram commutes,
giving a factorisation $(y^L,y^R)$ of the Yoneda embedding $y:J\morph{}[\opcat{J},\vsq]$, via the free $\cosmos$-cocompletion.
\begin{equation*}
  \begin{tikzcd}[cramped]
    \cat{C}_0 \ar[r,"J"] \ar[d,"y_0^L"] \ar[dd,bend right=50,"y_0"'] & \cat{C}_1 \ar[d,"y_1^L"] \ar[dd,bend left=50,"y_1"] \\
    \widehat{\cat{C}_0} \ar[r,"\lan_{\opcat{J}}^L"] \ar[d,"y_0^R"] & \overline{\cat{C}_1} \ar[d,"y_1^R"] \\
    {[\opcat{J},\vsq]}_0 \ar[r] & {[\opcat{J},\vsq]}_1
  \end{tikzcd}
\end{equation*}

So we now have an effectful category $\lan_{\opcat{J}}^L$ into which $J$ embeds.
The last thing to do is to check that it is closed, which follows by noting that $\lan_{\opcat{J}}$ is left adjoint to the functor which restricts presheaves along $J$, and taking bo-ff factorisations ensures that $\lan_{\opcat{J}}^L$ is also a left adjoint \cite{power_premonoidal}.

\section{$\vsq$-Profunctors}
In the previous Section we studied the notion of closure for effectful categories.
At this point we could stop and define \textit{pro}-effectful categories as ``closed effectful presheaf categories'' in analogy to promonoidal categories.
In fact, this definition is more subtle than it might first appear and it is certainly worth taking a little care.
In particular, given that the closed effectful embedding of any effectful category $J$ is given by the free \textit{tight} cocompletion $\lan_{\opcat{J}}^L$ and not the free cocompletion $[\opcat{J},\vsq]$, we must take care of what we mean by ``presheaf'' category here.
Furthermore, we would like to place pro-effectful categories on the same footing as promonoidal categories - as pseudomonoids in some form of bicategory of profunctors.

This will be the aim of this Section; to study the structure of $\vsq$-profunctors $P:\opcat{J}_\cat{D}\boxtimes J_\cat{C}\morph{}\vsq$.
By the following result we are able to unpack $P$ into a pair of $\cosmos$-profunctors together with a natural transformation between them.
The $\vsq$-natural transformations $\phi:P\Rightarrow Q$ can also be similarly unpacked.

\begin{prop}\label{prop:v2profs}
  Let $P:\opcat{J}_\cat{D}\otimes J_\cat{C}\morph{}\vsq$ be a $\vsq$-profunctor.
  Then $P$ is a triple of:
  \begin{enumerate}
    \item a $\cosmos$-profunctor $P_0:\opcat{\cat{D}_0}\otimes\cat{C}_0\morph{}\cosmos$,
    \item a $\cosmos$-profunctor $P_1:\opcat{\cat{D}_1}\otimes\cat{C}_1\morph{}\cosmos$,
    \item a $\cosmos$-natural transformation $\eta:P_0\Rightarrow P_1(\opcat{J}\otimes J)$.
  \end{enumerate}
  A $\vsq$-natural transformation $\phi:P\Rightarrow Q$ consists of $\cosmos$-natural transformations $\phi_0:P_0\Rightarrow Q_0$ and $\phi_1:P_1\Rightarrow Q_1$ such that $(\phi_1(\opcat{J}\otimes J)) \eta^P = \eta^Q \phi_0$.
\end{prop}
\begin{proof}
  This follows by applying a $\cosmos$-enriched version of a result by Power \cite[Prop. 24]{power_generic} to the functor category $[\opcat{J}_\cat{D}\otimes J_\cat{C},\vsq]\cong\Prof(J_\cat{C},J_\cat{D})$.
\end{proof}

The next proposition demonstrates that the coend of a $\vsq$-profunctor $P$ is given by the coends of $P_0$ and $P_1$ together with a canonical arrow between them.

\begin{prop}\label{prop:v2coend}
  Let $P:\opcat{J}\otimes J\morph{}\vsq$ be a $\vsq$-endoprofunctor.
  Then the coend $\int^c P(c,c)$ is given by the arrow $\int^c P_0(c,c) \morph{} \int^c P_1(c,c)$ induced by $\eta$ and the adjunction $y_J\dashv y^J$ in $\VProf$.
\end{prop}
\begin{proof}
  In Appendix \ref{proof:v2coend}.
\end{proof}

$\vsq$-endoprofunctors and the $\vsq$-natural transformations assemble into a $\vsq$-category $[\opcat{J}\otimes J,\vsq]\cong\Prof(J,J)$.
The category $\Prof(J,J)_0$ consists of the $\vsq$-profunctors and $\vsq$-natural transformations as outlined in Proposition \ref{prop:v2profs}, while $\Prof(J,J)_1$ has homs consisting of only the components $\phi_1$ of the natural transformations.
The identity on objects functor $\Prof(J,J)_0\morph{}\Prof(J,J)_1$ forgets the $\phi_0$ components.

As with any other category of endoprofunctors $\Prof(J,J)$ has a closed monoidal structure given by composition of the profunctors.
Given $P=(P_0,P_1,\eta_P)$ and $Q=(Q_0,Q_1,\eta_Q)$, their composition is given by $QP = (Q_0P_0,Q_1P_1,\eta_{QP})$ - we compose the underlying profunctors and take $\eta_{QP}$ to be given by $$\int^c Q(-,c)\otimes P(c,-) \xRightarrow{\int \eta_Q\otimes\eta_P} \int^{c\in\cat{C}_0} Q(J-,Jc)\otimes P(Jc,J-) \xRightarrow{y_J\dashv y^J} \int^{c\in\cat{C}_1} Q(J-,c)\otimes P(c,J-).$$

\subsection{Tight Profunctors}\label{sec:conical_prof}
In Section \ref{sec:closed_premon} we saw that effectful structure on $J:\cat{C}_0\morph{}\cat{C}_1$ induced a closed effectful structure on the free tight cocompletion of $J$.
It turns out that this effectful structure on $J$ is only a sufficient and not necessary condition for closed effectful structure on the free tight cocompletion of $J$.
Analogously to the case of monoidal categories where, in order for the presheaf category $\widehat{\cat{C}}$ to be closed monoidal it is only necessary that the category $\cat{C}$ is promonoidal \cite{day,day_thesis}, we only require $J$ to be a ``pro-effectful'' category.
To define these categories we need firstly to study the class of profunctors which factor through the tight cocompletion.
This will be the aim of this section.

To define pro-effectful categories we would like to replace the functors of a effectful category with profunctors, but we have a problem: we cannot consider arbitrary $\vsq$-profunctors $P:\opcat{J_\cat{D}}\otimes J_\cat{C}\morph{}\vsq$ because these assign arbitrary presheaves $\opcat{J_\cat{D}}\morph{}\vsq$ to objects of $J_\cat{C}$. These presheaves will not in general be contained in the free tight cocompletion.
Thus, we need a restricted class of profunctors, those that we call the \textit{tight} profunctors.

\begin{defn}[Tight $\vsq$-Profunctor]
  A tight $\vsq$-profunctor $P:J_\cat{C}\pmorph J_\cat{D}$ is a $\vsq$-functor $P:J_\cat{C}\morph{} \overline{J_\cat{D}}$, where $\overline{J_\cat{D}} \cong \lan^L_{\opcat{J_\cat{D}}}$ is the free tight cocompletion of $J_\cat{D}$.
\end{defn}

\begin{remark}
  Tight $\vsq$-profunctors can be unpacked component-wise analogously to Proposition \ref{prop:v2profs}, to see that they are precisely the $\vsq$-profunctors where $\eta$ is a natural \textit{isomorphism}.
\end{remark}

Similarly to how a profunctor $P:\cat{C}\pmorph\cat{D}$ is equivalently a cocontinuous functor between free cocompletions $\widehat{P}:\widehat{\cat{C}}\morph{}\widehat{\cat{D}}$, tight $\vsq$-profunctors are \textit{tightly} cocontinuous functors between free tight cocompletions.
\begin{defn}[Tightly Cocontinuous Functor]
  A $\vsq$-functor $F:J_\cat{C}\morph{}J_\cat{D}$ between tightly cocomplete categories is tightly cocontinuous if it preserves all tight colimits.
\end{defn}

\begin{thm}[\cite{kelly}]
  Let $\overline{J_\cat{C}}$ be the closure of $J_\cat{C}$ in $[\opcat{J_\cat{C}},\vsq]$ under tight colimits and write $y^L:J_\cat{C}\morph{}\overline{J_\cat{C}}$ for the inclusion.
  Then for tightly cocomplete $J_\cat{D}$, there is an equivalence
  \begin{equation*}
    \lan_{y^L}: [J_\cat{C},J_\cat{D}] \cong \mathsf{Cocont}_{tight}(\overline{J_\cat{C}},J_\cat{D})
  \end{equation*}
  where the right-hand is the category of tightly cocontinuous functors.
  This exhibits $\overline{J_\cat{C}}$ as the free tight cocompletion of $J_\cat{C}$.
\end{thm}

Indeed, $y^L$ is fully faithful so that there is a natural isomorphism $F\cong (\lan_{y^L}F)y^L$.
Consequently, we can think of a tight $\vsq$-profunctor $P:J_\cat{C}\morph{}\overline{J_\cat{D}}$ as a tightly cocontinuous functor $\tilde{P}:\overline{J_\cat{C}}\morph{}\overline{J_\cat{D}}$.
We can now define the following bicategory of tight $\vsq$-profunctors.
\begin{defn}
  Denote by $\vsqProf^\mathsf{Tight}$ the bicategory that has
  \begin{itemize}
    \item 0-cells the $\vsq$-categories $J:\cat{C}_0\morph{}\cat{C}_1$,
    \item 1-cells, $P:J_\cat{C}\pmorph J_\cat{D}$, the tight $\vsq$-profunctors $P:J_\cat{C}\morph{}\overline{J_\cat{D}}$,
    \item 2-cells the $\vsq$-natural transformations.
  \end{itemize}
  Composition of 1-cells is given by taking the left Kan extension along $y^L$ and composing the functors we obtain $Q\circ P = (\lan_{y^L}Q) P$.
\end{defn}

\begin{remark}
  We could also have defined tight $\vsq$-profunctors $J_\cat{C}\morph{}\overline{J_\cat{D}}$ as usual $\vsq$-profunctors $J_\cat{C}\morph{}[\opcat{J_\cat{D}},\vsq]$ that factorise via the embedding $y^R:\overline{J_\cat{D}}\morph{}[\opcat{J_\cat{D}},\vsq]$.
  Their usual composition as profunctors coincides (up to natural isomorphism) with the composition defined previously because $y^R$ is fully faithful and thus the unit of the Kan extension along $y^R$ is an isomorphism, $F\cong (\lan_{y^R}F)y^R$.
  It follows that
  \begin{equation*}
    Q\circ P = (\lan_y Q)P = (\lan_{y^Ry^L} Q)P \cong (\lan_{y^R}\lan_{y^L} Q)P = (\lan_{y^R}\lan_{y^L} Q)y^R P' \cong (\lan_{y^L}Q)P'.
  \end{equation*}
\end{remark}

\begin{remark}
  There is a more abstract but cleaner way to define the bicategory $\vsqProf^\mathsf{Tight}$, by noting that it is the Kleisli bicategory of a certain relative pseudomonad on $\vsqCat$.
  Relative pseudomonads were introduced in \cite{fiore} where it was also demonstrated that $\Prof$ is the Kleisli bicategory of the relative pseudomonad $\widehat{(\cdot)}$ of presheaves, which freely adds colimits by acting on 0-cells as $\cat{C}\mapsto\widehat{\cat{C}}$.
  Due to size issues, $\widehat{(\cdot)}$ is a relative pseudomonad and not just a plain pseudomonad: $\widehat{(\cdot)}$ sends small categories to locally small categories and so it is only a relative pseudomonad over the inclusion $\Cat\morph{}\mathsf{CAT}$ of the 2-category of small categories into the 2-category of locally small categories.

  In the same fashion there is a relative pseudomonad $\overline{(\cdot)}$ over the inclusion $\vsqCat\morph{}\vsq\text{-}\mathsf{CAT}$ which sends a small $\vsq$-category to its free tight cocompletion.
  It is then fairly straightforward to check that $\vsqProf^\mathsf{Tight}$ is the Kleisli bicategory of this relative pseudomonad and therefore also check that it is indeed a bicategory.
\end{remark}

$\vsqProf^\mathsf{Tight}$ has an interesting tensor product given by generalising the funny tensor product.

\begin{prop}[External Tensor Product]\label{prop:external_tensor}
  Let $J_\cat{C}$ and $J_\cat{D}$ be $\vsq$-categories and write $\overline{J_\cat{C}}$ and $\overline{J_\cat{D}}$ be their free tight cocompletions.
  Then there is a $\vsq$-functor
  \begin{equation}\label{eq:external_tensor}
    \hat{\otimes}:\overline{J_\cat{C}}\funny\overline{J_\cat{D}} \morph{} \overline{J_{\cat{C}\funny\cat{D}}}
  \end{equation}
  with components that act on objects as $(F\hat{\otimes} G)(c,d) :=  Fc\otimes Gd$.
\end{prop}
\begin{proof}
  In Appendix \ref{proof:external_tensor}.
\end{proof}

\begin{defn}[Funny Tensor Product of Tight $\vsq$-Profunctors]
  On categories the funny tensor acts like in $\vsqCat$.
  On tight $\vsq$-profunctors $P:J_\cat{A}\morph{}\overline{J_\cat{B}}$ and $Q:J_\cat{C}\morph{}\overline{J_\cat{D}}$ we define their funny tensor to be given by their funny tensor in $\vsqCat$ composed with the external tensor of free tight cocompletions \eqref{eq:external_tensor}:
  \begin{equation*}
    \begin{tikzcd}
      J_\cat{A}\funny J_\cat{C} \ar[r,"P\funny Q"] & \overline{J_\cat{B}}\funny\overline{J_\cat{D}} \ar[r,"\hat{\otimes}"] & \overline{J_{\cat{B}\funny\cat{D}}}
    \end{tikzcd}
  \end{equation*}
\end{defn}

\begin{thm}
  $\vsqProf^\mathsf{Tight}$ is a monoidal bicategory under the funny tensor product.
\end{thm}
\begin{proof}[Proof sketch]
  $\vsqProf^\mathsf{Tight}$ is the Kleisli bicategory of the relative pseudomonad $\overline{(\cdot)}$ that adds tight colimits.
  Under the funny tensor product on $\vsqCat$, this pseudomonad is monoidal, and therefore its Kleisli bicategory is also monoidal.
\end{proof}

\section{Pro-effectful Categories}\label{sec:pre_pro}
Finally in this Section we are in a position to define pro-effectful categories: as pseudomonoids in $\vsqProf^\mathsf{Tight}_{\funny}$, placing them on equal footing algebraically with monoidal, promonoidal and effectful categories.

\begin{defn}
  A pro-effectful category is a pseudomonoid in $\vsqProf^{\mathsf{Tight}}_{\funny}$.
  Explicitly, a pro-effectful category $J_\cat{C}$ is a $\vsq$-category equipped with a tensor product tight $\vsq$-profunctor $P:J_{\cat{C}\funny\cat{C}}\pmorph J_\cat{C}$ and a unit tight $\vsq$-profunctor $I:1\pmorph J_\cat{C}$, together with $\vsq$-natural isomorphisms $P(P\funny 1)\overset{\alpha}{\cong}P(1\funny P)$ and $P(I\funny 1)\overset{\lambda}{\cong}1\overset{\rho}{\cong}P(1\funny I)$ such that the triangle and pentagon equations hold.
\end{defn}

Like their effectful counterparts, pro-effectful categories also have an ``actegorical definition'' - they are a particular instance of a category equipped with an action by a promonoidal category.
This requires us firstly to weaken actegories to proactegories.

\begin{defn}[Proactegory]
  A left proactegory is a promonoidal category $(\cat{C}_0,P,I)$ and a category $\cat{C}_1$ equipped with a left proaction by $\cat{C}_0$, that is, a profunctor $L:\cat{C}_0\otimes\cat{C}_1\pmorph\cat{C}_1$ and natural isomorphisms
  \begin{equation*}
    \int^{X\in\cat{C}_1} L(A,B,X)\otimes L(X,C,D) \overset{a}{\cong} \int^{X\in\cat{C}_0} L(A,X,D)\otimes P(X,B,C), \quad
    \int^{X\in\cat{C}_0} L(A,X,B)\otimes I(X) \overset{l}{\cong} \cat{C}_1(A,B),
  \end{equation*}
  satisfying similar coherence diagrams as for an actegory.
  A biproactegory is simultaneously a left and right proactegory with an additional natural isomorphism
  \begin{equation*}
    \int^X R(D,X,C)\otimes L(X,A,B) \overset{b}{\cong} \int^X L(D,A,X)\otimes R(X,B,C)
  \end{equation*}
  satisfying similar coherences as for a biactegory.
\end{defn}

The following result generalises the equivalence between effectful categories and certain actegories \cite{levy_push} to the pro-effectful case.

\begin{prop}\label{prop:proeff_is_proact}
  A pro-effectful category is equivalently the following data:
  \begin{itemize}
    \item a promonoidal category $(\cat{C}_0,P_0,I_0)$,
    \item a category $\cat{C}_1$ with the same objects as $\cat{C}_0$ and an identity on objects functor $J:\cat{C}_0\morph{}\cat{C}_1$,
    \item left and right $\cat{C}_0$-proactions on $\cat{C}_1$, $P_1^L:\cat{C}_0\otimes\cat{C}_1\pmorph\cat{C}_1$ and $P_1^R:\cat{C}_1\otimes\cat{C}_0\pmorph\cat{C}_1$, which extend the canonical proactions of $\cat{C}_0$ on itself:
    \begin{equation}\label{eq:extend1}
      %\lan_{\opcat{J}\times 1\times 1} P_0 \cong P_1^L(1\times1\times J)
      \begin{tikzcd}[cramped]
        \cat{C}_0\otimes\cat{C}_0 \ar[r,"P_0","\shortmid" marking] \ar[d,"1\otimes yJ"',"\shortmid" marking]  & \cat{C}_0 \ar[d,"yJ","\shortmid" marking] \\
        \cat{C}_0\otimes\cat{C}_1 \ar[r,"P_1^L"',"\shortmid" marking] & \cat{C}_1
      \end{tikzcd}
      \hspace{2cm}
      \begin{tikzcd}[cramped]
        \cat{C}_0\otimes\cat{C}_0 \ar[r,"P_0","\shortmid" marking] \ar[d,"yJ \otimes 1"',"\shortmid" marking]  & \cat{C}_0 \ar[d,"yJ","\shortmid" marking] \\
        \cat{C}_1\otimes\cat{C}_0 \ar[r,"P_1^R"',"\shortmid" marking] & \cat{C}_1
      \end{tikzcd}
    \end{equation}
    \item a natural isomorphism $P_1^R(P_1^L\otimes 1) \cong P_1^L(1\otimes P_1^R)$ making $\cat{C}_1$ into a $\cat{C}_0$-$\cat{C}_0$-biproactegory.
  \end{itemize}
\end{prop}
\begin{proof}
  In Appendix \ref{proof:proeff_is_proact}.
\end{proof}

The next proposition generalises the equivalence between effectful categories and strong promonads \cite{jacobs_arrows,garner,roman_promonads} to the pro-effectful case.
The proof methods are related to those for promonoidal monads in \cite{day_monoidal_monads}.
\begin{prop}\label{prop:prostrong_promonad}
  A pro-effectful category is equivalently a prostrong promonad.
\end{prop}
\begin{proof}
  Take a prostrong promonad $T:\cat{C}\pmorph\cat{C}$.
  We will show we have the data of Proposition \ref{prop:proeff_is_proact}.

  $T$ has a Kleisli category in $\VProf$ and there is an identity on objects free functor $F:\cat{C}\morph{}\kl{T}$.
  By assumption $\cat{C}$ has a promonoidal structure $(P_0,I_0)$ and we can use the left and right prostrengths to define left and right proactions of $\cat{C}$ on $\kl{T}$.
  On objects the left proaction acts as $P_1^L(-,c,Fc') := \int^x \kl{T}(-,Fx)\otimes P_0(x,c,c')$ extending the canonical proaction on the centre, so that $\eqref{eq:extend1}$ commutes.
  Its action on homs is induced by the strength $\int^c P_0(-,-,c)\otimes T(c,-)\Rightarrow \int^c T(-,c)\otimes P_0(c,-,-)$.

  Conversely, suppose we are given a pro-effectful category $J:\cat{C}_0\morph{}\cat{C}_1$.
  Then $T(-,\bl):=\cat{C}_1(J-,J\bl)$ a promonad on $\cat{C}_0$ where the promonad multiplication and units are given by composition in $\cat{C}_1$.
  Moreover, $\cat{C}_1$ is precisely the Kleisli category of $T$.
  Now, since $J$ is pro-effectful, $\cat{C}_0$ is promonoidal and we are left to show that $T$ is prostrong over this structure.
  By Proposition \ref{prop:proeff_is_proact}, we have left and right proactions of $\cat{C}_1$ on $\cat{C}_0$ which preserve the canonical proaction on the centre and from these one can construct the prostrength of $T$. \qedhere

\end{proof}

Pro-effectful categories are also exactly what is required to place a closed effectful structure on the free tight cocompletion of a $\vsq$-category.
This generalises Day's theorem \cite{day,day_thesis} from monoidal to effectful categories, thus also generalising the result of Power on closed effectful embeddings of effectful categories \cite{power_structure,power_generic}.
The result follows by generalising the methods of Day's original proof, and from the folklore results regarding Day convolution for actegories, see \cite{janelidze_actions,campbell_skew}.
\begin{thm}\label{thm:proeff_closed}
  There is an equivalence between pro-effectful structures on $J$ and closed effectful structures on the free tight cocompletion $\overline{J} = \lan_{\opcat{J}}^L$.
\end{thm}
\begin{proof}
  In Appendix \ref{proof:proeff_closed}.
\end{proof}

Finally, we note some connections between pro-effectful categories and the premulticategories of Staton and Levy \cite{staton_premulticategories}, which generalise multicategories by dropping the interchange law.
Just as how promonoidal categories are examples of (co)multicategories \cite{day_centres}, pro-effectful categories are examples of (co)premulticategories.
Given a pro-effectful category $J:\cat{C}_0\morph{}\cat{C}_1$, there is a co-premulticategory $\mathfrak{C}$ with objects given by those of $\cat{C}_1$.
For $a,b\in \mathfrak{C}$ the class of arrows is given by $\mathfrak{C}(a;b):=\cat{C}_1(a,b)$ and for $a,b,c\in \mathfrak{C}$ the class of arrows is given by $\mathfrak{C}(a;b,c):=P_1(a,b,c)$.
The rest of the classes of arrows are defined inductively.

It is worth noting that there exist examples of pro-effectful categories which provide non-degenerate examples of premulticategories where the interchange law does not hold (in contrast to promonoidal and monoidal categories which are multicategories) and where the ``tensor'' is not representable (in contrast to monoidal and premonoidal categories).
For instance, the premonoidal optics introduced in the next section are an example of such a category.

\section{Premonoidal Optics}
In a seminal work on optics, Riley \cite{riley_optics} introduced the notion of ``effectful optics'': optics over the Kleisli category of a strong monad. These optics allow the emergence of side-effects, and extend the optics of pure functional programming to other programming languages with effects; with a similar purpose, Abou-Saleh et al \cite{abouSalehCGMS16} have introduced ``monadic lenses''. More recently, much applied category theory has been written about optics that create effects in different categories \cite{bolthedges,braithwaitehedges,clarke_profunctor,spivak2019generalized}.

We introduce a novel definition of optic over an effectful category that justifies this previous terminology: optics over the Kleisli category of a strong monad are particular cases of our effectful optics. We also introduce a proeffectful algebra over them that had been previously neglected.
In this section we will present the category of optics over a premonoidal category and outline its two tensor-like structures, analogous to those in Figure \ref{fig:optics_tensors}.

\begin{xlrbox}{premon_optics_hom1}
  \begin{tikzpicture}[baseline={([yshift=-.8ex]current bounding box.center)}]
    \node[3leggedpants,bluetube,wide] (pants) {};
    \node[tube,long,bluetube,anchor=top] (tube) at (pants.leftleg) {};
    \node[tube,long,bluetube,anchor=top] (tube2) at (pants.rightleg) {};
    \node[3leggedcopants,wide,bluetube,anchor=leftleg] (copants) at (tube.bot) {};
    \node[label] (g) at ([yshift=0.2cm]pants.center) {$g$};
    \node[label] (f) at ([yshift=-0.2cm]copants.center) {$f$};
    \begin{pgfonlayer}{edgelayer}
      \draw ([xshift=-0.075cm,yshift=0.15cm]f.center) to [out=90,in=-90] (copants.leftleg) to (pants.leftleg) to [out=90,in=-90] ([xshift=-0.075cm,yshift=-0.15cm]g.center);
      \draw ([xshift=0.075cm,yshift=0.15cm]f.center) to [out=90,in=-90] (copants.rightleg) to (pants.rightleg) to [out=90,in=-90] ([xshift=0.075cm,yshift=-0.15cm]g.center);
      \draw (copants.midleg) to (f) to (copants.belt);
      \draw (pants.midleg) to (g) to (pants.belt);
    \end{pgfonlayer}
    \node[system] at ([yshift=-0.15cm]copants.belt) {$a$};
    \node[system] at ([yshift=0.15cm]copants.midleg) {$b$};
    \node[system] at ([yshift=-0.15cm]pants.midleg) {$b'$};
    \node[system] at ([yshift=0.2cm]pants.belt) {$a'$};
  \end{tikzpicture}
\end{xlrbox}

\begin{xlrbox}{premon_optics_hom2}
  \begin{tikzpicture}[baseline={([yshift=-.8ex]current bounding box.center)}]
    \node[3leggedpants,wide] (pants) {};
    \node[tube,long,bluetube,anchor=top] (tube) at (pants.leftleg) {};
    \node[tube,long,bluetube,anchor=top] (tube2) at (pants.rightleg) {};
    \node[3leggedcopants,wide,anchor=leftleg] (copants) at (tube.bot) {};
    \node[label] (g) at ([yshift=0.2cm]pants.center) {$g$};
    \node[label] (f) at ([yshift=-0.2cm]copants.center) {$f$};
    \begin{pgfonlayer}{edgelayer}
      \draw ([xshift=-0.075cm,yshift=0.15cm]f.center) to [out=90,in=-90] (copants.leftleg) to (pants.leftleg) to [out=90,in=-90] ([xshift=-0.075cm,yshift=-0.15cm]g.center);
      \draw ([xshift=0.075cm,yshift=0.15cm]f.center) to [out=90,in=-90] (copants.rightleg) to (pants.rightleg) to [out=90,in=-90] ([xshift=0.075cm,yshift=-0.15cm]g.center);
      \draw (copants.midleg) to (f) to (copants.belt);
      \draw (pants.midleg) to (g) to (pants.belt);
    \end{pgfonlayer}
    \node[system] at ([yshift=-0.15cm]copants.belt) {$a$};
    \node[system] at ([yshift=0.15cm]copants.midleg) {$b$};
    \node[system] at ([yshift=-0.15cm]pants.midleg) {$b'$};
    \node[system] at ([yshift=0.2cm]pants.belt) {$a'$};
  \end{tikzpicture}
\end{xlrbox}

\begin{xlrbox}{vertical_tensor_central}
  \begin{tikzpicture}[baseline={([yshift=-.8ex]current bounding box.center)}]
    \node[3leggedpants,wide,bluetube] (pants) {};
    \node[tube,long,bluetube,anchor=top] (tube) at (pants.leftleg) {};
    \node[tube,long,bluetube,anchor=top] (tube2) at (pants.rightleg) {};
    \node[3leggedcopants,shortcrotch,wide,bluetube,anchor=leftleg] (copants) at (tube.bot) {};
    \node[3leggedpants,shortcrotch,wide,bluetube,anchor=belt] (pants2) at (copants.belt){};
    \node[tube,long,bluetube,anchor=top] (tube3) at (pants2.leftleg) {};
    \node[tube,long,bluetube,anchor=top] (tube4) at (pants2.rightleg) {};
    \node[3leggedcopants,wide,bluetube,anchor=leftleg] (copants2) at (tube3.bot) {};
    \node[label] (h) at ([yshift=0.2cm]pants.center) {$h$};
    \node[label] (g) at (copants.belt) {$g$};
    \node[label] (f) at ([yshift=-0.2cm]copants2.center) {$f$};
    \begin{pgfonlayer}{edgelayer}
      \draw ([xshift=-0.075cm,yshift=0.15cm]g.center) to [out=90,in=-90] (copants.leftleg) to (pants.leftleg) to [out=90,in=-90] ([xshift=-0.075cm,yshift=-0.15cm]h.center);
      \draw ([xshift=0.075cm,yshift=0.15cm]g.center) to [out=90,in=-90] (copants.rightleg) to (pants.rightleg) to [out=90,in=-90] ([xshift=0.075cm,yshift=-0.15cm]h.center);
      \draw (copants.midleg) to (g) to (pants2.midleg);
      \draw (pants.midleg) to (h) to (pants.belt);
      \draw ([xshift=-0.075cm,yshift=0.15cm]f.center) to [out=90,in=-90] (copants2.leftleg) to (pants2.leftleg) to [out=90,in=-90] ([xshift=-0.075cm,yshift=-0.15cm]g.center);
      \draw ([xshift=0.075cm,yshift=0.15cm]f.center) to [out=90,in=-90] (copants2.rightleg) to (pants2.rightleg) to [out=90,in=-90] ([xshift=0.075cm,yshift=-0.15cm]g.center);
      \draw (copants2.midleg) to (f) to (copants2.belt);
    \end{pgfonlayer}
    % \node[seam] at (tube.center) {};
    % \node[seam] at (tube2.center) {};
    \node[system] at ([yshift=-0.15cm]pants2.midleg) {$a'$};
    \node[system] at ([yshift=0.15cm]copants2.midleg) {$a$};
    \node[system] at ([yshift=0.15cm]copants.midleg) {$b$};
    \node[system] at ([yshift=-0.15cm]pants.midleg) {$b'$};
    \node[system] at ([yshift=0.2cm]pants.belt) {$c'$};
    \node[system] at ([yshift=-0.2cm]copants2.belt) {$c$};
  \end{tikzpicture}
\end{xlrbox}

\begin{xlrbox}{vertical_tensor_premon}
  \begin{tikzpicture}[baseline={([yshift=-.8ex]current bounding box.center)}]
    \node[3leggedpants,wide] (pants) {};
    \node[tube,long,bluetube,anchor=top] (tube) at (pants.leftleg) {};
    \node[tube,long,bluetube,anchor=top] (tube2) at (pants.rightleg) {};
    \node[3leggedcopants,shortcrotch,wide,anchor=leftleg] (copants) at (tube.bot) {};
    \node[3leggedpants,shortcrotch,wide,anchor=belt] (pants2) at (copants.belt){};
    \node[tube,long,bluetube,anchor=top] (tube3) at (pants2.leftleg) {};
    \node[tube,long,bluetube,anchor=top] (tube4) at (pants2.rightleg) {};
    \node[3leggedcopants,wide,anchor=leftleg] (copants2) at (tube3.bot) {};
    \node[label] (h) at ([yshift=0.2cm]pants.center) {$h$};
    \node[label] (g) at (copants.belt) {$g$};
    \node[label] (f) at ([yshift=-0.2cm]copants2.center) {$f$};
    \begin{pgfonlayer}{edgelayer}
      \draw ([xshift=-0.075cm,yshift=0.15cm]g.center) to [out=90,in=-90] (copants.leftleg) to (pants.leftleg) to [out=90,in=-90] ([xshift=-0.075cm,yshift=-0.15cm]h.center);
      \draw ([xshift=0.075cm,yshift=0.15cm]g.center) to [out=90,in=-90] (copants.rightleg) to (pants.rightleg) to [out=90,in=-90] ([xshift=0.075cm,yshift=-0.15cm]h.center);
      \draw (copants.midleg) to (g) to (pants2.midleg);
      \draw (pants.midleg) to (h) to (pants.belt);
      \draw ([xshift=-0.075cm,yshift=0.15cm]f.center) to [out=90,in=-90] (copants2.leftleg) to (pants2.leftleg) to [out=90,in=-90] ([xshift=-0.075cm,yshift=-0.15cm]g.center);
      \draw ([xshift=0.075cm,yshift=0.15cm]f.center) to [out=90,in=-90] (copants2.rightleg) to (pants2.rightleg) to [out=90,in=-90] ([xshift=0.075cm,yshift=-0.15cm]g.center);
      \draw (copants2.midleg) to (f) to (copants2.belt);
    \end{pgfonlayer}
    % \node[seam] at (tube.center) {};
    % \node[seam] at (tube2.center) {};
    \node[system] at ([yshift=-0.15cm]pants2.midleg) {$a'$};
    \node[system] at ([yshift=0.15cm]copants2.midleg) {$a$};
    \node[system] at ([yshift=0.15cm]copants.midleg) {$b$};
    \node[system] at ([yshift=-0.15cm]pants.midleg) {$b'$};
    \node[system] at ([yshift=0.2cm]pants.belt) {$c'$};
    \node[system] at ([yshift=-0.2cm]copants2.belt) {$c$};
  \end{tikzpicture}
\end{xlrbox}

\begin{xlrbox}{vertical_tensor_central_unit}
  \begin{tikzpicture}[baseline={([yshift=-.8ex]current bounding box.center)}]
    \node[tube,longer,bluetube] (t) {};
    \node[label] (f) at (t.center) {$f$};
    \begin{pgfonlayer}{edgelayer}
      \draw (t.top) to (t.bot);
    \end{pgfonlayer}
    \node[system] at ([yshift=0.15cm]t.top) {$a'$};
    \node[system] at ([yshift=-0.15cm]t.bot) {$a$};
  \end{tikzpicture}
\end{xlrbox}

\begin{xlrbox}{vertical_tensor_premon_unit}
  \begin{tikzpicture}[baseline={([yshift=-.8ex]current bounding box.center)}]
    \node[tube,longer] (t) {};
    \node[label] (f) at (t.center) {$f$};
    \begin{pgfonlayer}{edgelayer}
      \draw (t.top) to (t.bot);
    \end{pgfonlayer}
    \node[system] at ([yshift=0.15cm]t.top) {$a'$};
    \node[system] at ([yshift=-0.15cm]t.bot) {$a$};
  \end{tikzpicture}
\end{xlrbox}

Suppose we fix a premonoidal category $\cat{C}$ and write $J:Z\cat{C}\morph{}\cat{C}$ for the inclusion of the centre.
There is a $\vsq$-category $\opt(J)$ with objects given by pairs $\vb{a}:=(a,a')$ of those of $J$, i.e.\ pairs of those of the underlying premonoidal category $\cat{C}$.
The homs are given by $\int^{xy} Z\cat{C}(a,x\otimes b\otimes y) \otimes Z\cat{C}(x\otimes b' \otimes y, a') \morph{} \int^{xy\in Z\cat{C}} \cat{C}(a,x\otimes b\otimes y) \otimes \cat{C}(x\otimes b' \otimes y, a')$, as in Figure \ref{fig:premon_optics_homs}.
Thus $\opt(J)_0 = \opt(Z\cat{C})$ is the usual category of optics over the centre and $\opt(J)_1 = \opt_{Z\cat{C}}(\cat{C})$ is the category of optics given by the action of the centre $Z\cat{C}$ on the whole premonoidal category $\cat{C}$.
The identity on objects functor $\opt(Z\cat{C})\morph{}\opt_{Z\cat{C}}(\cat{C})$ is the one induced by $J$.

\begin{thm}\label{thm:optJ_promon}
  $\opt(J)$ is a promonoidal $\vsq$-category.
  The $\vsq$-profunctors forming the tensor product $P:\opt(J)\otimes\opt(J)\pmorph\opt(J)$ and unit $I:1\pmorph\opt(J)$ have components given in Figures \ref{fig:premon_optics_vertical_unit} and \ref{fig:premon_optics_vertical_tensor}. These are explicitly,
  \begin{align*}
    P_0(\vb{c},\vb{a},\vb{b}) & = \int^{xx'yy'} Z\cat{C}(c,x\otimes a\otimes x')\otimes Z\cat{C}(x\otimes a'\otimes x',y\otimes b\otimes y')\otimes Z\cat{C}(y\otimes b'\otimes y',c'), \\
    P_1(\vb{c},\vb{a},\vb{b}) & = \int^{xx'yy'\in Z\cat{C}} \cat{C}(c,x\otimes a\otimes x')\otimes \cat{C}(x\otimes a'\otimes x',y\otimes b\otimes y')\otimes \cat{C}(y\otimes b'\otimes y',c'), \\
    & \hspace{2cm} I_0(\vb{a}) = Z\cat{C}(a,a'), \hspace{1.5cm} I_1(\vb{a}) = \cat{C}(a,a').
  \end{align*}
\end{thm}
\begin{proof}
  In Appendix \ref{proof:optJ_promon}.
\end{proof}

\noindent
\begin{minipage}[b]{0.5\textwidth}\centering
  \scalebox{0.9}{\xusebox{premon_optics_hom1}} $ \morph{} $ \scalebox{0.9}{\xusebox{premon_optics_hom2}}
  \captionof{figure}{Homs.}
  \label{fig:premon_optics_homs}
  \scalebox{0.9}{\xusebox{vertical_tensor_central_unit}} $\morph{}$ \scalebox{0.9}{\xusebox{vertical_tensor_premon_unit}}
  \captionof{figure}{Promonoidal unit $I$.}
  \label{fig:premon_optics_vertical_unit}
\end{minipage}%
\begin{minipage}[b]{0.5\textwidth}\centering
  \scalebox{0.9}{\xusebox{vertical_tensor_central}} $ \morph{} $ \scalebox{0.9}{\xusebox{vertical_tensor_premon}}
  
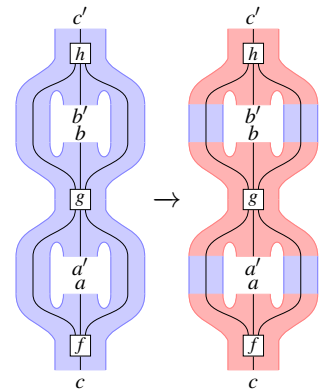
\captionof{figure}{Promonoidal tensor $P$.}
  \label{fig:premon_optics_vertical_tensor}
\end{minipage}

Now let us turn our attention to another tensor-like structure on $\opt_{1_{Z\cat{C}}}(J)$, this one induced by the premonoidal structure on $\cat{C}$.

\begin{xlrbox}{horizontal_tensor_central}
  \begin{tikzpicture}[baseline={([yshift=-.8ex]current bounding box.center)}]
    \node[5leggedpants,wider,bluetube] (pants) {};
    \node[tube,long,bluetube,anchor=top] (tube) at (pants.leftleftleg) {};
    \node[tube,long,bluetube,anchor=top] at (pants.midleg) {};
    \node[tube,long,bluetube,anchor=top] at (pants.rightrightleg) {};
    \node[5leggedcopants,wider,bluetube,anchor=leftleftleg] (copants) at (tube.bot) {};
    \node[label] (f) at ([yshift=-0.3cm]copants.center) {$f$};
    \node[label] (g) at ([yshift=0.3cm]pants.center) {$g$};
    \begin{pgfonlayer}{edgelayer}
      \draw (pants.belt) to (copants.belt);
      \draw ([xshift=-0.09cm]f.center) to [out=90,in=-90] (copants.leftleftleg) to (pants.leftleftleg) to [out=90,in=-90] ([xshift=-0.09cm]g.center);
      \draw ([xshift=0.09cm]f.center) to [out=90,in=-90] (copants.rightrightleg) to (pants.rightrightleg) to [out=90,in=-90] ([xshift=0.09cm]g.center);\
      \draw ([xshift=-0.05cm,yshift=0.1cm]f.center) to [out=90,in=-90] (copants.leftleg);
      \draw ([xshift=0.05cm,yshift=0.1cm]f.center) to [out=90,in=-90] (copants.rightleg);
      \draw ([xshift=-0.05cm,yshift=-0.1cm]g.center) to [out=-90,in=90] (pants.leftleg);
      \draw ([xshift=0.05cm,yshift=-0.1cm]g.center) to [out=-90,in=90] (pants.rightleg);
    \end{pgfonlayer}
    \node[system] at ([yshift=0.15cm]copants.leftleg) {$a$};
    \node[system] at ([yshift=-0.15cm]pants.leftleg) {$a'$};
    \node[system] at ([yshift=0.15cm]copants.rightleg) {$b$};
    \node[system] at ([yshift=-0.15cm]pants.rightleg) {$b'$};
    \node[system] at ([yshift=-0.15cm]copants.belt) {$c$};
    \node[system] at ([yshift=0.15cm]pants.belt) {$c'$};
  \end{tikzpicture}
\end{xlrbox}

\begin{xlrbox}{horizontal_tensor_premon}
  \begin{tikzpicture}[baseline={([yshift=-.8ex]current bounding box.center)}]
    \node[5leggedpants,wider] (pants) {};
    \node[tube,long,bluetube,anchor=top] (tube) at (pants.leftleftleg) {};
    \node[tube,long,bluetube,anchor=top] at (pants.midleg) {};
    \node[tube,long,bluetube,anchor=top] at (pants.rightrightleg) {};
    \node[5leggedcopants,wider,anchor=leftleftleg] (copants) at (tube.bot) {};
    \node[label] (f) at ([yshift=-0.3cm]copants.center) {$f$};
    \node[label] (g) at ([yshift=0.3cm]pants.center) {$g$};
    \begin{pgfonlayer}{edgelayer}
      \draw (pants.belt) to (copants.belt);
      \draw ([xshift=-0.09cm]f.center) to [out=90,in=-90] (copants.leftleftleg) to (pants.leftleftleg) to [out=90,in=-90] ([xshift=-0.09cm]g.center);
      \draw ([xshift=0.09cm]f.center) to [out=90,in=-90] (copants.rightrightleg) to (pants.rightrightleg) to [out=90,in=-90] ([xshift=0.09cm]g.center);\
      \draw ([xshift=-0.05cm,yshift=0.1cm]f.center) to [out=90,in=-90] (copants.leftleg);
      \draw ([xshift=0.05cm,yshift=0.1cm]f.center) to [out=90,in=-90] (copants.rightleg);
      \draw ([xshift=-0.05cm,yshift=-0.1cm]g.center) to [out=-90,in=90] (pants.leftleg);
      \draw ([xshift=0.05cm,yshift=-0.1cm]g.center) to [out=-90,in=90] (pants.rightleg);
    \end{pgfonlayer}
    \node[system] at ([yshift=0.15cm]copants.leftleg) {$a$};
    \node[system] at ([yshift=-0.15cm]pants.leftleg) {$a'$};
    \node[system] at ([yshift=0.15cm]copants.rightleg) {$b$};
    \node[system] at ([yshift=-0.15cm]pants.rightleg) {$b'$};
    \node[system] at ([yshift=-0.15cm]copants.belt) {$c$};
    \node[system] at ([yshift=0.15cm]pants.belt) {$c'$};
  \end{tikzpicture}
\end{xlrbox}

\begin{xlrbox}{horizontal_tensor_premon_unit}
  \begin{tikzpicture}[baseline={([yshift=-.8ex]current bounding box.center)}]
    \node[tube] (t) {};
    \node[tube,bluetube,long,anchor=top] (t2) at (t.bot) {};
    \node[tube,anchor=top] (t3) at (t2.bot) {};
    \node[label] (f) at (t2.center) {$f$};
    \begin{pgfonlayer}{edgelayer}
      \draw (t.top) to (t3.bot);
    \end{pgfonlayer}
    \node[system] at ([yshift=0.15cm]t.top) {$a'$};
    \node[system] at ([yshift=-0.15cm]t3.bot) {$a$};
  \end{tikzpicture}
\end{xlrbox}

\begin{thm}\label{thm:opt_is_proeff}
  $\opt(J)$ is a pro-effectful category.
  The tight $\vsq$-profunctors forming the tensor product $P:\opt(J)\otimes\opt(J)\morph{}\overline{\opt(J)}$ and unit $I:1\morph{}\overline{\opt(J)}$ have components which act on objects as,
  \begin{equation}\label{eq:proeff}
    \begin{aligned}
      P_0(\vb{c},\vb{a},\vb{b}) = P_1(\vb{c},\vb{a},\vb{b}) & = \int^{xyz} Z\cat{C}(c,x\otimes a\otimes y \otimes b\otimes z) \otimes Z\cat{C}(x\otimes a'\otimes y\otimes b'\otimes z,c'), \\
      I_0(\vb{a}) = I_1(\vb{a}) & = Z\cat{C}(a,a').
    \end{aligned}
  \end{equation}
\end{thm}
\begin{proof}
  In Appendix \ref{proof:opt_is_proeff}.
\end{proof}

On the homs of $Z\cat{C}$, $P_0$ and $I_0$ act in the expected way, essentially by nesting of optics.
On the homs of $\cat{C}$, $P_1$ and $I_1$ act somewhat unusually.
Formally non-central optics are sent to natural transformations between left Kan extensions of the expressions in \eqref{eq:proeff}, that is between presheaves of the form:
\begin{align*}
  (\lan_{\opcat{J}\otimes J} P_0)(\vb{c},\vb{a},\vb{b}) & \cong \int^{wvxyz} \cat{C}(c,Jw) \otimes Z\cat{C}(w,x\otimes a\otimes y \otimes b\otimes z) \otimes Z\cat{C}(x\otimes a'\otimes y\otimes b'\otimes z,v) \otimes \cat{C}(Jv,c') \\
  & \cong \int^{xyz\in Z\cat{C}} \cat{C}(c,x\otimes a\otimes y \otimes b\otimes z) \otimes \cat{C}(x\otimes a'\otimes y\otimes b'\otimes z,c') \\
  (\lan_{\opcat{J}\otimes J} I_0)(\vb{a}) & \cong \int^{xy} \cat{C}(a,Jx)\otimes Z\cat{C}(x,y) \otimes \cat{C}(Jy,a') \cong \int^{x\in Z\cat{C}} \cat{C}(a,x) \otimes \cat{C}(x,a')
\end{align*}
This justifies thinking of the pro-effectful structure as having the components described in Figures \ref{fig:premon_optics_horz_tensor} and \ref{fig:premon_optics_horz_unit}.

\noindent
\begin{minipage}[b]{0.6\textwidth}
  \centering
  \scalebox{0.9}{\xusebox{horizontal_tensor_central}} $ \morph{} $ \scalebox{0.9}{\xusebox{horizontal_tensor_premon}}
  
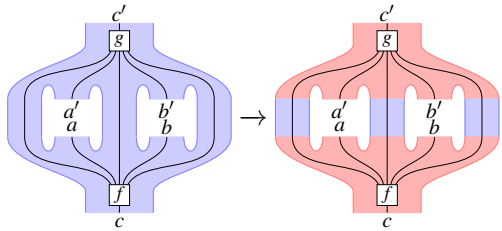
\captionof{figure}{Pro-effectful tensor.}
  \label{fig:premon_optics_horz_tensor}
\end{minipage}%
\begin{minipage}[b]{0.4\textwidth}
  \centering
  \scalebox{0.9}{\xusebox{vertical_tensor_central_unit}} $ \morph{} $ \scalebox{0.9}{\xusebox{horizontal_tensor_premon_unit}}
  
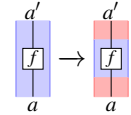
\captionof{figure}{Pro-effectful unit.}
  \label{fig:premon_optics_horz_unit}
\end{minipage}

% \newpage
\subsubsection*{Acknowledgements}
The authors want to thank Matt Earnshaw and Matt Wilson for discussion; the authors also want to thank the anonymous reviewers at ACT23 for multiple suggestions that improved this article.
James Hefford is supported by University College London and the EPSRC [grant number EP/L015242/1]. Mario Román is supported by the European Union through the ESF Estonian IT Academy research measure (2014-2020.4.05.19-0001).

\bibliographystyle{eptcs}
\bibliography{bibliography}

% \newpage
\appendix

\section{Proofs}

\subsection{Proof of Theorem \ref{thm:v2cat}}\label{proof:v2cat}
\begin{proof}[Proof sketch]
  The behaviour of the funny tensor on functors is encapsulated by the following cube.
  
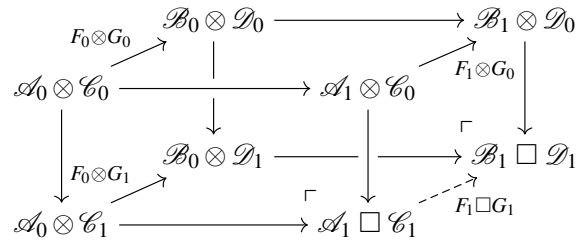
\begin{equation*}
    \begin{tikzcd}[row sep=small, column sep=small,cramped]
      & \cat{B}_0\otimes\cat{D}_0 \ar[rr] \ar[dd] & & \cat{B}_1\otimes\cat{D}_0 \ar[dd] \\
      \cat{A}_0\otimes\cat{C}_0 \ar[rr, crossing over] \ar[dd] \ar[ur,"F_0\otimes G_0"] & & \cat{A}_1\otimes\cat{C}_0  \ar[ur,"F_1\otimes G_0"'] \\
      & \cat{B}_0\otimes\cat{D}_1 \ar[rr] & & \cat{B}_1\funny\cat{D}_1 \arrow[ul, phantom, "\ulcorner", very near start] \\
      \cat{A}_0\otimes\cat{C}_1  \ar[rr] \ar[ur,"F_0\otimes G_1"] & & \cat{A}_1\funny\cat{C}_1  \ar[from=uu, crossing over] \ar[ur,dashed,"F_1\funny G_1"'] \arrow[ul, phantom, "\ulcorner", very near start] \\
    \end{tikzcd}
  \end{equation*}
  Functoriality of $\square$ on 1-cells follows by pasting of cubes and the uniqueness of the arrows induced by the pushout.

  Explicitly, we have $(\alpha:F\Rightarrow F')\funny (\beta:G\Rightarrow G')$ has components $(\alpha\funny\beta)^0_{cd} = (\alpha^0_c,\beta^0_d)$ and $(\alpha\funny\beta)^1_{cd} = (\alpha^1_c,\beta^1_d) = (J\alpha^0_c,J\beta^0_d)$.
  Naturality of this transformation follows from naturality of $\alpha$ and $\beta$ and from the centrality of the components.
\end{proof}

\subsection{Proof of Proposition \ref{prop:v2coend}}\label{proof:v2coend}
\begin{proof}
  Suppose we have a $\vsq$-extranatural family $w_c:P(c,c)\morph{}d$.
  Then we have the following commutative diagram:
  \begin{equation*}
    \begin{tikzcd}[row sep=small, column sep=small,cramped]
      \cat{C}_1(c,c') \otimes P_1(c',c) \ar[rrr] \ar[ddd] & & & P_1(c',c') \ar[ddd,"w_{c'}^1"] \\
      & \cat{C}_0(c,c') \otimes P_0(c',c) \ar[d] \ar[r] \ar[lu,"J\otimes\eta_{c'c}"] & P_0(c',c') \ar[d,"w_{c'}^0"] \ar[ru,"\eta_{c'c'}"'] & \\
      & P_0(c,c) \ar[r,"w_c^0"'] \ar[ld,"\eta_{cc}"'] & d_0 \ar[dr,"d"] & \\
      P_1(c,c) \ar[rrr,"w_c^1"'] & & & d_1
    \end{tikzcd}
  \end{equation*}
  In particular, the families $w_c^0:P_0(c,c)\morph{}d_0$ and $w_c^1:P_1(c,c)\morph{}d_1$ are $\cosmos$-extranatural and thus factorise via their respective coends giving arrows $\int^c P_0(c,c)\morph{}d_0$ and $\int^c P_1(c,c) \morph{}d_1$ making the obvious diagrams commute.
  Now note that the arrows $P_0(c,c)\morph{\eta_{cc}}P_1(c,c)\morph{\text{copr}_c}\int^c P_1(c,c)$ are $\cosmos$-extranatural, this induces a arrow $\int^c P_0(c,c) \morph{} \int^c P_1(c,c)$.
\end{proof}

\subsection{Proof of Proposition \ref{prop:external_tensor}}\label{proof:external_tensor}
\begin{proof}
  To give \eqref{eq:external_tensor} is to give a pair of functors such that the following square commutes:
  \begin{equation*}
  \begin{tikzcd}[cramped]
    \widehat{\cat{C}_0}\otimes\widehat{\cat{D}_0} \ar[r] \ar[d,"\hat{\otimes}_0"'] & \overline{\cat{C}_1}\funny\overline{\cat{D}_1} \ar[d,"\hat{\otimes}_1"] \\
    \widehat{\cat{C}_0\otimes\cat{D}_0} \ar[r,"\lan_{\opcat{J_{\cat{C}\funny\cat{D}}}}^L"']& \overline{\cat{C}_1\funny\cat{D}_1}
  \end{tikzcd}
  \end{equation*}
  The tensor
  $\hat{\otimes}_0$ acts on objects following the formula $(F\hat{\otimes} G)(c,d) :=  Fc\otimes Gd$ and on morphisms in the obvious way.
  The tensor $\hat{\otimes}_1$ also acts following the same formula, note that a morphism of $\overline{\cat{C}_1}\funny\overline{\cat{D}_1}$ is a free composition of natural transformations $\alpha:\lan_{\opcat{J_\cat{C}}}F\Rightarrow \lan_{\opcat{J_\cat{C}}}F'$ and $\beta:\lan_{\opcat{J_\cat{D}}}G\Rightarrow \lan_{\opcat{J_\cat{D}}}G'$ with $(1;\beta)(\alpha;1)\neq(\alpha;1)(1;\beta)$ in general.
  Each such arrow induces a natural transformation $\lan_{\opcat{J_\cat{C\funny\cat{D}}}}(F\otimes G)\Rightarrow\lan_{\opcat{J_\cat{C\funny\cat{D}}}}(F'\otimes G')$, for instance:
  % \begin{align*}
  %   \lan_{\opcat{J_\cat{B\funny\cat{D}}}}(F\times G) & \cong \int^{bd} (\cat{B}\funny\cat{D})((-,=), J_{\cat{B}\funny\cat{D}}(b,d))\times Fb\times Gd \\
  %   & \cong \int^{b'd} (\cat{B}\funny\cat{D})((-,=), i_0(b',d))\times (\lan_{\opcat{J_\cat{B}}}F)b' \times Gd \\
  %   & \morph{\int 1\times\alpha\times1} \int^{b'd} (\cat{B}\funny\cat{D})((-,=), i_0(b',d))\times (\lan_{\opcat{J_\cat{B}}}F')b' \times Gd \\
  %   & \cong \int^{bd'} (\cat{B}\funny\cat{D})((-,=), i_1(b,d'))\times F'b \times (\lan_{\opcat{J_\cat{D}}}G)d' \\
  %   & \morph{\int 1\times1\times\beta} \int^{bd'} (\cat{B}\funny\cat{D})((-,=), i_1(b,d'))\times F'b \times (\lan_{\opcat{J_\cat{D}}}G')d' \\
  %   & \cong \lan_{\opcat{J_\cat{B\funny\cat{D}}}}(F'\times G')
  % \end{align*}
  \begin{align*}
  \lan_{\opcat{J_\cat{C\funny\cat{D}}}}(F\otimes G) & \cong \lan_{\opcat{i_1}}(\lan_{\opcat{J_\cat{C}}}F \otimes G) \xRightarrow{\lan_{\opcat{i_1}}(\alpha\otimes1)} \lan_{\opcat{i_1}}(\lan_{\opcat{J_\cat{C}}}F' \otimes G) \cong \lan_{\opcat{i_0}}(F' \otimes \lan_{\opcat{J_\cat{D}}}G) \\
  & \xRightarrow{\lan_{\opcat{i_0}}(1\otimes\beta)} \lan_{\opcat{i_0}}(F' \otimes \lan_{\opcat{J_\cat{D}}}G') \cong \lan_{\opcat{J_\cat{C\funny\cat{D}}}}(F'\otimes G')
  \end{align*}
\end{proof}

\subsection{Proof of Proposition \ref{prop:proeff_is_proact}}\label{proof:proeff_is_proact}
\begin{proof}
  Fix a pro-effectful category $(J,P,I)$.
  $J$ is a $\vsq$-category so we have two categories $\cat{C}_0$ and $\cat{C}_1$ with the same objects and an identity on objects functor $J:\cat{C}_0\morph{}\cat{C}_1$.

  The tight $\vsq$-profunctor $P:J_{\cat{C}\funny\cat{C}}\pmorph J_{\cat{C}}$ consists of a profunctor $P_0:\cat{C}_0\otimes\cat{C}_0\pmorph\cat{C}_0$ and a functor $P_1:\cat{C}_1\funny\cat{C}_1\morph{}\overline{\cat{C}_1}$.
  Similarly, the tight $\vsq$-profunctor $I:1\pmorph J_{\cat{C}}$ consists of presheaves $I_0:\opcat{\cat{C}_0}\morph{}\cosmos$ and $I_1=\lan_{\opcat{J}}I_0:\opcat{\cat{C}_1}\morph{}\cosmos$.
  $(P_0,I_0)$ induce a promonoidal structure on $\cat{C}_0$.

  $P_1$ induces the left and right proactions of $\cat{C}_0$ on $\cat{C}_1$.
  Starting with the left proaction, $P_1$ induces a functor $y_1^R P_1 i_1=:P_1^L:\cat{C}_0\otimes\cat{C}_1\morph{}\widehat{\cat{C}_1}$.
  It follows that:
  \begin{equation*}
    P^L_1(1\otimes J) = y_1^R P_1 i_1(1\otimes J) = y_1^R P_1 J_{\cat{C}\funny\cat{C}} = \lan_{\opcat{J}\otimes 1\otimes 1} P_0
  \end{equation*}
  showing that \eqref{eq:extend1} commutes and that the left proaction extends the canonical one on $\cat{C}_0$.
  A similar argument holds for the right proaction.

  Suppose now that we start with the data specified in the proposition.
  The equalities \eqref{eq:extend1} together with the universal property of the pushout induce a functor $P_1:\cat{C}_1\funny\cat{C}_1\morph{}\overline{\cat{C}_1}$
  \begin{equation*}
    \begin{tikzcd}[cramped]
      \cat{C}_0\otimes\cat{C}_0 \ar[r,"J\otimes1"] \ar[d,"1\otimes J"'] & \cat{C}_1\otimes\cat{C}_0  \ar[d,"i_0"] \ar[ddr,bend left,"P^R_1"] & \\
      \cat{C}_0\otimes\cat{C}_1 \ar[r,"i_1"'] \ar[rrd,bend right,"P^L_1"'] & \cat{C}_1\funny\cat{C}_1 \ar[rd,dashed,"P_1"]& \\
      & & \overline{\cat{C}_1}
    \end{tikzcd}
  \end{equation*}
  and it follows that $P_1 J_{\cat{C}\funny\cat{C}} = \lan_{\opcat{J}}^L P_0$ making $(P_0,P_1)$ the components of a tight $\vsq$-profunctor $P:J_{\cat{C}\funny\cat{C}}\pmorph J_{\cat{C}}$.
  The presheaf $I_0:\opcat{\cat{C}_0}\morph{}\cosmos$ together with its Kan extension $I_1:=\lan_{\opcat{J}} I_0$ give the components of a $\vsq$-profunctor $I:1\pmorph J$.
  Checking all the coherences is a long but ultimately routine calculation.
\end{proof}

\subsection{Proof of Theorem \ref{thm:proeff_closed}}\label{proof:proeff_closed}
\begin{proof}
  Suppose $J:\cat{C}_0\morph{}\cat{C}_1$ is a pro-effectful category.
  We will show that $\lan^L_{\opcat{J}}:\widehat{\cat{C}_0}\morph{}\overline{\cat{C}_1}$ is a closed premonoidal category.
  Since $\cat{C}_0$ is promonoidal, $\widehat{\cat{C}_0}$ is closed monoidal under Day convolution.

  As for the premonoidal structure on $\overline{\cat{C}_1}$: on objects it is the same as on $\widehat{\cat{C}_0}$.
  On morphisms, suppose we are given a $\eta:F\Rightarrow G$ in $\overline{\cat{C}_1}$.
  Then we have a $\eta:\lan_{\opcat{J}}F\Rightarrow\lan_{\opcat{J}}G$ and we can describe the left hand part of the premonoidal structure by
  \begin{align*}
    \lan_{\opcat{J}}(F\star F')(-) & \cong \int^{abc} \cat{C}_1(-,Jc)\otimes P_0(c,a,b)\otimes Fa \otimes F'b \cong \int^{ab} P^R_1(-,Ja,b)\otimes Fa \otimes F'b \\
    & \cong \int^{bc} P^R_1(-,c,b)\otimes (\lan_{\opcat{J}}F)(c) \otimes F'b \\
    & \xRightarrow{\int\eta} \int^{bc} P^R_1(-,c,b)\otimes (\lan_{\opcat{J}}G)(c) \otimes F'b \cong \lan_{\opcat{J}}(G\star F')(-)
  \end{align*}
  and similarly for the right hand part.
  It is easily seen that $\lan^L_{\opcat{J}}$ factorises through the centre of this premonoidal structure.

  The internal-hom of the left-closed premonoidal structure, $[G,-]:\overline{\cat{C}_1}\morph{}\widehat{\cat{C}_0}$ is given by
  \begin{equation*}
    [G,H](a) \cong \int_{cd} \cosmos \bigg( P_1^L(c,a,d), \cosmos\big( (\lan_{\opcat{J}}G)(d), (\lan_{\opcat{J}}H)(c) \big) \bigg)
  \end{equation*}
  while the right-closed structure is similar, replacing $P_1^L$ with $P_1^R$.
  In both cases, checking we have the required adjunction is a matter of standard coend calculus.
  % \begin{align*}
  %   \widehat{\cat{C}_0} (F,[G,H]) & \cong \int_a \cosmos \bigg( Fa, \int_{cd} \cosmos \big( P_1^L(c,a,d)\otimes (\lan_{\opcat{J}}G)(d),(\lan_{\opcat{J}}H)(c) \big) \bigg) \\
  %   & \cong \int_{acd} \cosmos \bigg( Fa\otimes P_1^L(c,a,d) \otimes (\lan_{\opcat{J}}G)(d), (\lan_{\opcat{J}}H)(c) \bigg) \\
  %   & \cong \int_{c} \cosmos \bigg( \int^{ab} Fa\otimes P_1^L(c,a,Jb) \otimes Gb, (\lan_{\opcat{J}}H)(c) \bigg) \\
  %   & \cong \int_c \cosmos \bigg( \lan_{\opcat{J}}(F\star G)(c), (\lan_{\opcat{J}}H)(c) \bigg) \\
  %   & \cong \overline{\cat{C}_1}(F\otimes G,H)
  % \end{align*}

  Suppose now that $\lan_{\opcat{J}}^L$ is a closed effectful category.
  Then it follows that $\widehat{\cat{C}_0}$ is a closed monoidal category because:
  \begin{align*}
    \widehat{\cat{C}_0} \left( -,[G,\lan_{\opcat{J}}^L(=)] \right) & \cong \overline{\cat{C}_1}\left(\lan^L_{\opcat{J}}(-)\boxtimes G,\lan_{\opcat{J}}^L(=)\right) = \overline{\cat{C}_1}\left(\lan^L_{\opcat{J}}(-\otimes G),\lan_{\opcat{J}}^L(=)\right) \\
    & \cong \widehat{\cat{C}_0}\left(-\otimes G,{\opcat{J}}^*(\lan_{\opcat{J}}^L(=))\right) \cong \widehat{\cat{C}_0}\left(-\otimes G,=\right)
  \end{align*}
  where ${\opcat{J}}^*$ is the right adjoint to $\lan_{\opcat{J}}^L$, both of which are ioo.
  Therefore $\cat{C}_0$ is a promonoidal category.

  The left $\cat{C}_0$-proaction on $\cat{C}_1$ is given by $P_1^L(-,a,b) := y_0^L(a)\boxtimes y_1^L(b) = \boxtimes i_1 (y_0^L(a), y_1^L(b))$ and similarly for the right.
  These extend the canonical proaction because:
  \begin{align*}
    P_1^L(-,a,Jb) = \boxtimes i_1 (y_0^L(a), y_1^L(Jb)) & = \boxtimes i_1 (y_0^L(a), \lan^L_{\opcat{J}}y_0^L(b)) = \boxtimes i_1 (1\otimes_\cosmos\lan^L_{\opcat{J}}) (y_0^L(a),y_0^L(b)) \\
    & = \lan^L_{\opcat{J}}\mathord{\otimes} (y_0^L(a),y_0^L(b)) = \lan^L_{\opcat{J}} P_0(-,a,b)
  \end{align*}
  where we have written the monoidal operation $\otimes$ on $\widehat{\cat{C}_0}$ and the premonoidal operation $\boxtimes$ on $\overline{\cat{C}_1}$ with prefix notation.
\end{proof}

\subsection{Proof of Theorem \ref{thm:optJ_promon}}\label{proof:optJ_promon}
\begin{proof}
$J$ has commutative left and right actions by the monoidal $\vsq$-category $1_{Z\cat{C}}:Z\cat{C}\morph{}Z\cat{C}$.
Consider the $\vsq$-category $\Tamb(J)$ of Tambara modules on $J$ \cite{pastro_street,clarke_profunctor}, whose objects are the $\vsq$-endoprofunctors $P:J\pmorph J$ equipped with left and right strengths over the action by $1_{Z\cat{C}}$. The morphisms are the bistrong $\vsq$-natural transformations.
We can use Proposition \ref{prop:v2profs} to unpack $\Tamb(J)$ into two $\cosmos$-categories and an identity on objects functor, $\Tamb(J)_0\morph{} \Tamb(J)_1$.
The objects of $\Tamb(J)_0$ and $\Tamb(J)_1$ are the bistrong endoprofunctors $P:J\pmorph J$ which are equivalently triples $(P_0:Z\cat{C}\pmorph Z\cat{C}, P_1:\cat{C}\pmorph \cat{C},\eta: P_0\Rightarrow P_1(\opcat{J}\otimes J))$.
$\Tamb(J)_0$ has arrows $\phi:P\Rightarrow Q$ given by pairs $(\phi_0:P_0\Rightarrow Q_0,\phi_1:P_1\Rightarrow Q_1)$ while $\Tamb(J)_1$ has only the $\phi_1$ as arrows.

It is known that the category of Tambara modules is equivalent to the presheaf category of the category of optics \cite{pastro_street,clarke_profunctor}, which in this particular case implies $[\opcat{\opt(J)},\vsq]\cong \Tamb(J)$.
The $\vsq$-category $\opt(J)$ has objects given by pairs $\vb{a}=(a,a')$ of $\opt(J)$ and homs given by
\begin{equation*}
  \opt(J)(\vb{a},\vb{b}) = \int^{xy\in 1_{Z\cat{C}}} J(a,x\otimes b\otimes y)\otimes J(x\otimes b'\otimes y, a')
\end{equation*}
where $J(-,-):= Z\cat{C}(-,-)\morph{}\cat{C}(-,-)$ is the hom of $J$ as a $\vsq$-category and the coend is taken in this fully enriched setting.
By Proposition \ref{prop:v2coend} this coend is given by the following arrow.
\begin{equation*}
  \int^{xy} Z\cat{C}(a,x\otimes b\otimes y)\otimes Z\cat{C}(x\otimes b'\otimes y, a') \morph{} \int^{xy\in Z\cat{C}} \cat{C}(a,x\otimes b\otimes y)\otimes \cat{C}(x\otimes b'\otimes y, a')
\end{equation*}
As a result, the identity on objects functor equivalent to $\opt(J)$ is given by $\opt(Z\cat{C})\morph{}\opt_{Z\cat{C}}(\cat{C})$ as expected.

Now, since $\Tamb(J)$ has a closed monoidal structure given by composition of the profunctors, there is an induced promonoidal structure on $\opt(J)$.
To arrive at the explicit expressions claimed in the Theorem, take objects $\vb{a}$ and $\vb{b}$ of $\opt(J)$ and consider the tensor (i.e. composition as profunctors) of the associated representable presheaves.
\begin{equation*}
  (y_{\vb{a}} \otimes y_{\vb{b}})(-) \cong \int^{wxyz\in 1_{Z\cat{C}}} J(-,w\otimes a\otimes x)\otimes J(w\otimes a'\otimes x, y\otimes b\otimes z) \otimes J(y\otimes b'\otimes z,-)
\end{equation*}
This can be unpacked by Proposition \ref{prop:v2coend} to give the result.

Finally note that the unit of the monoidal structure on $\Tamb(J)$ is $1_J:J\pmorph J$, which is $(1_{Z\cat{C}},1_{\cat{C}},\eta:1_{Z\cat{C}}\Rightarrow y^J y_J)$.
\end{proof}

\subsection{Proof of Theorem \ref{thm:opt_is_proeff}}\label{proof:opt_is_proeff}
\begin{proof}
  The free tight cocompletion of $\opt(J)$ is given by $[\opcat{\opt(Z\cat{C})},\cosmos]\morph{} \overline{\opt_{Z\cat{C}}(\cat{C})}$.
  We will show that this is a closed effectful category and then by Theorem \ref{thm:proeff_closed} we will be done.

  Start by considering the effectful category $\opcat{J}\otimes J:\opcat{Z\cat{C}}\otimes Z\cat{C} \morph{} \opcat{\cat{C}}\otimes\cat{C}$.
  The free tight cocompletion of this category is $\lan^L_{\opcat{J}\otimes J}: \Prof(Z\cat{C}) \morph{} \overline{\Prof(\cat{C})}$ which is closed effectful.
  The domain is the duoidal category $\Prof(Z\cat{C})$ of endoprofunctors on $Z\cat{C}$ and it has a closed monoidal structure given by Day convolution over the monoidal structure of $Z\cat{C}$:
  \begin{equation}\label{eq:central_horz_tensor}
    P * Q := \int^{aa'bb'} Z\cat{C}(-,a\otimes a') \otimes P(a,b) \otimes Q(a',b') \otimes Z\cat{C}(b\otimes b',-)
  \end{equation}
  The premonoidal structure on $\overline{\Prof(\cat{C})}$ is given on objects by \eqref{eq:central_horz_tensor}, and on homs, given a $\eta:P\Rightarrow P'$ in $\overline{\Prof(\cat{C})}$ (that is, a $\eta:\lan_{\opcat{J}\otimes J}P\Rightarrow \lan_{\opcat{J}\otimes J}P'$) the left side of the premonoidal structure is given by:
  \begin{align*}
    \lan_{\opcat{J}\otimes J} (P*Q) \cong & \int^{aa'bb'} \cat{C}(-,J(a\otimes a')) \otimes P(a,b) \otimes Q(a',b') \otimes \cat{C}(J(b\otimes b'),-) \\
    \cong & \int^{a'b'\in Z\cat{C}, cd\in\cat{C}} \cat{C}(-,c\rtimes a')) \otimes (\lan_{\opcat{J}\otimes J}P)(c,d) \otimes Q(a',b') \otimes \cat{C}(d\rtimes b',-) \\
    \xRightarrow{\int\eta} & \int^{a'b'\in Z\cat{C}, cd\in\cat{C}} \cat{C}(-,c\rtimes a')) \otimes (\lan_{\opcat{J}\otimes J}P')(c,d) \otimes Q(a',b') \otimes \cat{C}(d\rtimes b',-) \\
    \cong & \ \lan_{\opcat{J}\otimes J} (P'*Q)
  \end{align*}
  Since $\lan^L_{\opcat{J}\otimes J}$ is a left adjoint, it follows that it is a closed effectful category.

  There is a $\vsq$-category $\Tamb(Z\cat{C}) \morph{} \overline{\Tamb(\cat{C})}$ with objects given by the Tambara modules on $Z\cat{C}$.
  The homs of $\Tamb(Z\cat{C})$ are the bistrong natural transformations while the homs of $\overline{\Tamb(\cat{C})}$ are the bistrong natural transformations between the left Kan extensions along $\opcat{J}\otimes J$ of the Tambara modules.
  This $\vsq$-category inherits a closed effectful structure from $\lan^L_{\opcat{J}\otimes J}$ given by a certain quotient of \eqref{eq:central_horz_tensor} which acts to normalise the duoidal structure on $\Prof(Z\cat{C})$ \cite{garner, earnshaw}.

  Finally note that the presheaf category of optics is equivalent to the category of Tambara modules, $\widehat{\opcat{\opt(Z\cat{C})}} \cong \mathsf{Tamb}(Z\cat{C})$ \cite{clarke_profunctor}, and we can finally check that we also have $\overline{\opt_{Z\cat{C}}(\cat{C})} \cong \overline{\Tamb(\cat{C})}$.
\end{proof}

\end{document}

%% file: preamble.tex
\usepackage[utf8]{inputenc}
\usepackage{amsmath}
\usepackage{amsthm}
\usepackage{amssymb}
\usepackage{physics}
\usepackage{stmaryrd}
\usepackage{tikz}
\usepackage{tikz-cd}
\usepackage{circuitikz}
\usepackage[numbers]{natbib}
\usepackage{hyperref}
\usepackage[T1]{fontenc}
\usepackage{url}
\usepackage{multirow}
\usepackage[font=small]{caption,subcaption}

\newcommand{\morph}[1]{\xrightarrow{#1}}

\newcommand{\pmorph}{\relbar\joinrel\mapstochar\joinrel\rightarrow}

\newcommand{\cat}[1]{\mathcal{#1}}

\newcommand{\opcat}[1]{#1^\textrm{op}}

\newcommand{\set}{\mathsf{Set}}

\newcommand{\cosmos}{\mathcal{V}}
\newcommand{\vsq}{\mathcal{V}^2}
\newcommand{\vsqCat}{\vsq\text{-}\Cat}
\newcommand{\vsqProf}{\vsq\text{-}\Prof}

\newcommand{\Cat}{\mathsf{Cat}}
\newcommand{\Prof}{\mathsf{Prof}}

\newcommand{\VCat}{\cosmos\text{-}\mathsf{Cat}}
\newcommand{\VProf}{\cosmos\text{-}\mathsf{Prof}}
\newcommand{\Tamb}{\mathsf{Tamb}}
\newcommand{\funny}{\mathrel{\square}}
\newcommand{\lan}{\textrm{Lan}}

\newcommand{\opt}{\mathsf{Optic}}
\newcommand{\kl}[1]{\mathsf{Kl}_{#1}}
\newcommand{\bl}{\mathord{=}}

\DeclareFontFamily{U}{min}{}
\DeclareFontShape{U}{min}{m}{n}{<-> udmj30}{}

\makeatletter
\newcommand{\xRightarrow}[2][]{\ext@arrow 0359\Rightarrowfill@{#1}{#2}}
\makeatother

\theoremstyle{definition}
\newtheorem{defn}{Definition}
\theoremstyle{plain}
\newtheorem{prop}{Proposition}
\theoremstyle{plain}

\theoremstyle{plain}
\newtheorem{thm}{Theorem}
\theoremstyle{plain}

\theoremstyle{remark}
\newtheorem*{remark}{Remark}
\theoremstyle{definition}